\documentclass[11pt]{amsart} 
\usepackage{amsmath,amsthm}
\usepackage{xypic}
\usepackage{enumerate}
\usepackage{algorithm}
\usepackage{algpseudocode}

\usepackage{breqn}

\usepackage{lscape}
\usepackage{tikz-cd}
\usepackage{tikz}
\usepackage{color} 
\definecolor{darkred}{rgb}{1,0,0} 
\definecolor{darkgreen}{rgb}{0,1,0}
\definecolor{darkblue}{rgb}{0,0,1}
\usepackage{hyperref}
\hypersetup{colorlinks,
linkcolor=darkblue,
filecolor=darkblue,
urlcolor=darkred,
citecolor=darkgreen}

\usepackage{amscd}
\usepackage{amsfonts,amssymb,latexsym} 
\usepackage[all]{xy} 
\xyoption{arc}
\CompileMatrices

\newcommand{\udots}{\mathinner{\mskip1mu\raise1pt\vbox{\kern7pt\hbox{.}}
\mskip2mu\raise4pt\hbox{.}\mskip2mu\raise7pt\hbox{.}\mskip1mu}}

\newcommand{\SH}{{\mathcal{H}}}

\newcommand{\SL}{{\mathcal{L}}}
\newcommand{\SM}{{\mathcal{M}}}

\newcommand{\SP}{{\mathcal{P}}}

\newcommand{\ST}{{\mathcal{T}}}

\newcommand{\KK}{\mathbb{K}}
\newcommand{\LL}{\mathbb{L}}

\newcommand{\ZZ}{\mathbb{Z}}
\newcommand{\NN}{\mathbb{N}}

\newcommand{\QQ}{\mathbb{Q}}

\newcommand{\Sym}{\operatorname{Sym}}
\newcommand{\Alt}{\operatorname{Alt}}

\newcommand{\rk}{\operatorname{rk}}

\newcommand{\Jac}{\operatorname{Jac}}

\newcommand{\coeff}{\mathop{\mathrm{coeff}}}

\newcommand{\VarK}{\mathcal{V}ar_{\KK}}
\newcommand{\ChowK}{\mathcal{CM}_{\KK}}

\newcommand{\SLrings}{\SL_{\op{rings}}}
\newcommand{\SLlambda}{\SL_{\lambda-\op{rings}}}
\newcommand{\SLdrings}{\SLrings^{\op{div}}}
\newcommand{\SLdlambda}{\SLlambda^{\op{div}}}
\newcommand{\depth}{\op{depth}}

\newcommand{\op}{\operatorname}

\newtheorem{proposition}{Proposition}[section]
\newtheorem{theorem}[proposition]{Theorem}
\newtheorem{definition}[proposition]{Definition}

\newtheorem{lemma}[proposition]{Lemma}
\newtheorem{conjecture}[proposition]{Conjecture}
\newtheorem{corollary}[proposition]{Corollary}
\newtheorem{remark}[proposition]{Remark}

\newtheorem{example}[proposition]{Example}

\numberwithin{equation}{section}

\title{Simplification of $\lambda$-ring expressions in the Grothendieck ring of Chow motives}
\author[D. Alfaya]{David Alfaya}

\address{D. Alfaya, 
\newline\indent
Department of Applied Mathematics and Institute for Research in Technology, ICAI School of Engineering, Comillas Pontifical University, C/Alberto Aguilera 25, 28015 Madrid, Spain}

\email{dalfaya@comillas.edu}

\keywords{lambda-rings, symbolic computations of motives, Chow motives, moduli spaces, Higgs bundles moduli space}

\subjclass[2020]{13D15, 68W30, 19E08, 14C35, 14D20, 14H60}

\begin{document}

\begin{abstract}
The Grothendieck ring of Chow motives admits two natural opposite $\lambda$-ring structures, one of which is a special structure allowing the definition of Adams operations on the ring. In this work I present algorithms which allow an effective simplification of expressions that involve both $\lambda$-ring structures, as well as Adams operations. In particular, these algorithms allow the symbolic simplification of algebraic expressions in the sub-$\lambda$-ring of motives generated by a finite set of curves into polynomial expressions in a small set of motivic generators. As a consequence, the explicit computation of motives of some moduli spaces is performed, allowing the computational verification of some conjectural formulas for these spaces.
\end{abstract}

\maketitle

\tableofcontents

\section{Introduction}
A $\lambda$-ring is an abelian unital ring endowed with a set of maps $\lambda^n$ for each natural number which satisfy the identities 
$$\lambda^0(x)=1, \quad \quad \quad \lambda^1(x)=x, \quad \quad \quad \lambda^n(x+y)=\sum_{i=0}^n \lambda^i(x) \lambda^{n-i}(x)$$
These rings, first introduced by Grothendieck, generalize simultaneously several common structures in mathematics. For example, binomial coefficients $\binom{x}{n}$ give a $\lambda$-ring structure on $\ZZ$ and symmetric powers and exterior powers of bundles induce different $\lambda$-ring structures on the K-theory of a manifold.

In algebraic geometry, $\lambda$-ring structures play an important role in invariant computations and, more precisely, in computations of motives of algebraic varieties. In addition to its relation with K-theory, symmetric powers of varieties induce a natural $\lambda$-ring structure on both Grothendieck's ring of varieties $K_0(\VarK)$ and several spaces of motives, such as the ring of Chow motives $\ChowK$ and its Grothendieck ring $K_0(\ChowK)$ \cite{H07}. In particular, it is common for the motive of a variety to be described as an algebraic expression in the corresponding $\lambda$-ring.

In particular, this type of $\lambda$-expressions appear naturally in the computation of motivic classes of different types of moduli spaces. For example, the virtual classes in $K_0(\ChowK)$ of moduli spaces of vector bundles of rank $n=2,3$ over a curve $X$ with fixed degree and fixed determinant are obtained as algebraic functions of the classes of symmetric powers of the curve $C$ and the Lefschetz motive $\LL$ \cite{GPHS11,S14,Lee18,GL20}. This also holds for moduli spaces of Higgs bundles \cite{GPHS11} and, conjecturally, for the virtual class of moduli stacks of vector bundles of any rank in the completion of the Grothendieck ring of varieties \cite{BD07}.

Nevertheless, manipulating, simplifying and comparing these type of expressions can be difficult. Expanding $\lambda$ operations on an arbitrary $\lambda$-expression in a general $\lambda$-ring is not always entirely possible, as expressions like $\lambda^n(xy)$ or $\lambda^i(\lambda^j(x))$ might not be simplifiable unless the $\lambda$-structure satisfies additional axioms (e.g., providing a special $\lambda$-ring structure). Moreover, expressions involved in motive computations sometimes involve more than one $\lambda$-structure (typically, a $\lambda$-structure, its opposite structure and its associated Adams operations). Even when possible, expanding these operations directly from the axioms can be computationally demanding.

In this paper, an algorithm for simplifying arbitrary algebraic expressions involving a $\lambda$-ring, its opposite $\lambda$-structure and its associated Adams operations is obtained under the condition that the opposite $\lambda$-ring is special. More precisely, the algorithm transforms any such expression into an integral polynomial in $\lambda$-powers of elementary generators of the original expression (see Algorithm \ref{alg:simpAlg} and Theorems \ref{thm:exprSimplification} and \ref{thm:mainAlgorithm}).

\begin{theorem}
Let $\SLlambda=(+,-,\cdot,0,1)\cup \{\lambda^n,\sigma^n,\psi^n\}_{n\in \NN}$ be the language of rings with two $\lambda$-structures $\lambda$ and $\sigma$ and Adams operations $\psi$, and let $\ST$ be the theory of $\lambda$-rings $(R,\lambda)$ whose opposite $\lambda$-structure, $\sigma$ is special. Then for each expression $\varphi(x_1,\ldots,x_n)$ in the language $\SLlambda$, there exists a unique integral polynomial $P_\varphi^\lambda\in \ZZ[x_{1,1},\ldots,x_{1,d_1},\ldots,x_{n,1},\ldots,x_{n,d_n}]$ such that
$$\ST \models \forall x_1,\ldots, x_n \, \varphi(x_1,\ldots,x_n)=P_{\varphi}^\lambda(\lambda^1(x_1),\ldots, \lambda^{d_1}(x_1),\ldots,\lambda^1(x_n),\ldots,\lambda^{d_n}(x_n))$$
Moreover, for torsion free rings, Algorithm \ref{alg:simpAlg} finds the polynomial $P_{\varphi}^\lambda$ from every expression $\varphi$.
\end{theorem}

In particular, this is the situation happening in the Grothendieck ring of Chow motives, where the $\lambda$-ring is induced by symmetric powers of varieties. This $\lambda$-ring is not special itself, so, in principle, simplification of arbitrary expressions would be challenging, but in \cite{H07} it is proven that its opposite $\lambda$-structure is indeed special, so the proposed algorithm can be applied. In particular, incorporating the properties of motives of curves, we obtain a way to reduce any motivic expression generated by a finite set of curves into an integral polynomial in a finite set of motivic generators (See Theorem \ref{thm:exprSimpMotives} and Algorithm \ref{alg:motiveSimpAlg})

\begin{theorem}
Let $X_1,\ldots,X_n$ be smooth projective curves over a field $\KK$ of characteristic $0$. Then any expression $P_\varphi$ in the sub-$(\lambda,\sigma,\psi)$-ring of the completion of the Grothendieck ring of Chow motives $\hat{K}_0(\ChowK)$ spanned by the motives of the curves and the Lefschetz motive $\LL=[\mathbb{A}^1]$ is simplified by Algorithm \ref{alg:motiveSimpAlg} into an integral polynomial $P_\varphi$ such that
$$\varphi(\LL,[X_1],\ldots,[X_n])=P_\varphi(\LL,a_{1,1},\ldots,a_{1,g_1},\ldots,a_{n,g_n})$$
where $a_{i,k}=\lambda^k([h^1(X_i)])$.
\end{theorem}

As an application of this theorem, we will verify computationally a conjectural formula stated by Mozgovoy \cite{Moz12} for the moduli space of twisted Higgs bundles on a curve in low rank, genus and degree of the twisting line bundle. To check the result, we will compare it with the recent motivic formula for such moduli spaces obtained in \cite{AO21}. In principle, both formulas are very different and involve nontrivial expressions in the $\lambda$-ring of motives. Moreover, in the case of Mozgovoy's conjecture, the expression involves the application of Adams operations for the opposite structure of the symmetric $\lambda$-ring structure on several $\lambda$-ring expressions. Using the proposed algorithm, both formulas are converted into integral polynomials on the same set of generators, which depend solely on the curve. The resulting polynomials are verified to coincide, thus proving the conjectural formula for the motive.

The structure of the paper is the following. We start recalling the definitions of $\lambda$-structure, opposite $\lambda$-structure and Adams operations, as well as the algebraic relations between them (Section \ref{section:lambdaRings}). Section \ref{section:algorithm} includes the main simplification theorems for $\SLlambda$-terms and the proposed simplification algorithm for abstract $\lambda$-rings $(R,\lambda)$ whose opposite $\lambda$-structure is special. These results are then particularized in Section \ref{section:KChow} for the $\lambda$-structure on the Grothendieck ring of Chow motives induced by taking symmetric powers of varieties. As an application, in Section \ref{section:Higgs} we will apply the proposed algorithm to prove computationally that Mozgovoy's conjectural formula for the motive of the moduli space of $L$-twisted Higgs bundles holds in low genus, rank and degree. Finally, as an annex, some of the obtained explicit simplified polynomial formulas for the motives of such moduli spaces of twisted Higgs bundles are included.

\noindent\textbf{Acknowledgments.} 
This research was funded by MICINN grant PID2019-108936GB-C21. I would like to thank André Oliveira for useful discussions.

\section{$\lambda$-rings and Adams operations}
\label{section:lambdaRings}
We will start by a brief reminder about $\lambda$-ring structures (c.f. \cite{Knut73}, \cite{Gri19}). Let $R$ be an abelian unital ring.
\begin{definition}
\label{def:lambdaRing}
A $\lambda$-ring structure on $R$ is set of maps (not necesarily homomorphisms) $\lambda^i: R\longrightarrow R$ indexed by natural numbers $i\in \mathbb{N}$ which satisfy the following properties.
\begin{enumerate}
\item For all $x\in R$, $\lambda^0(x)=1$ and $\lambda^1(x)=x$.
\item For all $x,y\in R$ and all $n\in \mathbb{N}$
$$\lambda^n(x+y)= \sum_{i=0}^n \lambda^i(x) \lambda^{n-i}(y)$$
\end{enumerate}
\end{definition}
Alternatively, we can codify the information of the $\lambda^i$ maps into a generating series
$$\lambda_t(x)= \sum_{i\ge 0} \lambda^i(x) t^i$$
In this context, condition (1) ensures that for each $x\in R$, $\lambda(x)\in 1+tR[[t]]$ and it is therefore an invertible series and condition (2) is equivalent to stating that for each $x,y\in R$ the following equality holds in $1+tR[[t]]$
$$\lambda_t(x+y)=\lambda_t(x)\lambda_t(y)$$
i.e. that the map $\lambda_t: (R,+)\longrightarrow (1+tR[[t]], \cdot)$ is a group homomorphism. From these properties one can deduce that for each $\lambda$-ring structure $\lambda$ on $R$, there exists an ``opposite'' $\lambda$-ring structure $\sigma$ given by the following relation
\begin{equation}
\label{eq:oppositeLambda}
\sigma_t(x)=\sum_{i\ge 0} \sigma^i(x)t^i = \left( \sum_{i\ge 0} \lambda^i(x) (-t)^i \right)^{-1}=(\lambda_{-t}(x))^{-1}
\end{equation}
In particular, as $1=\lambda_t(0)=\lambda_t(x)\lambda_t(-x)$, we have
$$\sigma_t(-x)=\lambda_{-t}(x)$$
so
\begin{equation}
\label{eq:sigma-x}
\sigma^n(-x)=(-1)^n\lambda^n(x)
\end{equation}

\begin{definition}
\label{def:specialLambdaRing}
A $\lambda$-ring is special if, moreover, for all $x,y\in R$ and all $n,m\in \mathbb{N}$
$$\lambda^n(xy)=P_n\left (\lambda^1(x),\ldots,\lambda^n(x),\lambda^1(y),\ldots,\lambda^n(y) \right)$$
$$\lambda^n(\lambda^m(x))=P_{n,m} \left (\lambda^1(x),\ldots,\lambda^{nm}(x) \right)$$
where $P_n\in \mathbb{Z}(X_1,\ldots,X_n,Y_1,\ldots,Y_n)$ and $P_{n,m}\in \mathbb{Z}(X_1,\ldots,X_{nm})$ are certain universal polynomials called Grothendieck polynomials, defined as follows. If
$$s_{n,m}(X_1,\ldots,X_m)=\sum_{1\le j_1 < \ldots < j_n \le m} \prod_{k=1}^n X_{j_k}$$
is the elementary symmetric polynomial of degree $n$ the $m$ variables $X_1,\ldots, X_m$, then
$$P_{n,m}\left (s_{1,nm}(\{X_i\}),\ldots,s_{nm,nm}(\{X_i\})\right)=\op{coeff}_{t^n} \prod_{1\le j_1<\ldots<j_m \le nm} \left(1+t \prod_{k=1}^m X_{j_k}\right)$$
$$P_n \left (s_{1,n}(\{X_i\}),\ldots, s_{n,n}(\{X_i\}),s_{1,n}(\{Y_i\}),\ldots,s_{n,n}(\{Y_i\}) \right) = \op{coeff}_{t^n} \prod_{i,j=1}^n (1+tX_iY_j)$$
\end{definition}

If $\lambda$ is special, then a certain set of well-behaved endomorphisms of the $\lambda$-ring $(R,\lambda)$ called Adams operations can be constructed as follows.

\begin{definition}
\label{def:AdamsOperations}
Let $\lambda$ be a $\lambda$-ring structure on the ring $R$. Let $N_n\in \ZZ[s_1,\ldots,s_n]$ be the $n$-th Hirzebruch-Newton polynomial, defined as follows. It is the unique polynomial such that for each $m\in \NN$
$$\sum_{i=1}^m X_i^n = N_n(s_{1,m}(\{X_i\}),\ldots,s_{n,m}(\{X_i\})) \in \ZZ[X_1,\ldots,X_m]$$
Then for each $n\in \NN$, define the $n$-th Adams operation of the $\lambda$-ring $(R,\lambda)$ as
$$\psi_n(x)=N_n(\lambda^1(x),\ldots,\lambda^n(x))$$
\end{definition}

\begin{proposition}[c.f. {\cite[Theorem 9.2]{Gri19}}]
\label{prop:psiHomo}
If $\lambda$ is special, then 
\begin{enumerate}
\item For every $j\in \NN$, $\psi_j:(R;\lambda) \to (R,\lambda)$ is a $\lambda$-ring homomorphism.
\item For every $x\in R$, $\psi_1(x)=x$.
\item For every $i,j\in \NN$, $\psi_i \circ \psi_j = \psi_{ij}$.
\end{enumerate}
\end{proposition}

From this point on, let $(R,\lambda)$ be a $\lambda$-ring, and assume that the opposite $\lambda$-ring structure, which will be called $\sigma$, is special. Let $\psi_n$ be the Adams operations associated with such special $\lambda$-structure $\sigma$. Let us recall some algebraic relations between these three families of maps, $\lambda^n$,  $\sigma^n$ and $\psi_n$. Although most of these relations are well known in the literature (c.f. \cite{Knut73}, \cite{Gri19}), a proof of each of them is included as, in addition to demonstrating the corresponding results, they present some effective recursive methods for the computation of the maps which will be used in the main simplification algorithm.

First of all, solving the triangular system of equations defined by Equation \eqref{eq:oppositeLambda} and taking into account that the opposite structure of $\sigma$ is $\lambda$ yields the following.

\begin{proposition}
\label{prop:oppositePol}
For each $n>0$ there exists a degree $n$ polynomial $P^{op}_n(x_1,\ldots,x_n)\in \ZZ[x_1,\ldots,x_n]$ such that for all $x\in R$
$$\sigma^n(x)=P^{op}_n(\lambda^1(x),\ldots,\lambda^n(x))$$
$$\lambda^n(x)=P^{op}_n(\sigma^1(x),\ldots,\sigma^n(x))$$
and the polynomials $P^{op}_n$ satisfy the following recursive relation.
\begin{align*}
P^{op}_0&=1\\
P^{op}_n&=\sum_{i=0}^{n-1} P^{op}_i x_{n-i} (-1)^{n-i+1} \quad \quad \forall n>1
\end{align*}
\end{proposition}

\begin{proof}
As the opposite $\lambda$-structure of $\sigma$ is $\lambda$, it is only necessary to prove that there exists a set of polynomials satisfyig the given recursive relation and such that for each $x\in R$
$$\sigma^n(x)=P^{op}_n(\lambda^1(x),\ldots,\lambda^n(x))$$
We will prove this by induction on $n$. For $n=0$, the axioms imply that $\sigma^0(x)=1=\lambda^0(x)$ for each $x\in R$, so $P^{op}_0=1$ as stated. Let $n>0$ and suppose that the polynomials $P^{op}_i$ exist for each $i< n$ and stisfy the requiered properties. For every $x\in R$ we have
$$1=\left( \sum_{i\ge 0} \sigma^i(x) t^i \right) \left ( \sum_{j\ge 0} \lambda^j(x) (-t)^j \right) = \sum_{n\ge 0} \sum_{i=0}^n \left(\sigma^i(x) \lambda^{n-i}(x) (-1)^{n-i} \right) t^k$$
Thus, for every $n>0$
$$\sigma^n(x) + \sum_{i=0}^{n-1} \sigma^i(x) \lambda^{n-i}(x) (-1)^{n-i}=0$$
And, therefore
$$\sigma^n(x) = \sum_{i=0}^{n-1} \sigma^i(x) \lambda^{n-i}(x)(-1)^{n-i+1}$$
Now we can apply induction and write $\sigma^i(x) = P^{op}_i(\lambda^1(x),\ldots,\lambda^i(x))$, so
$$\sigma^n(x) = \sum_{i=0}^{n-1} P^{op}_i(\lambda^1(x),\ldots,\lambda^i(x)) \lambda^{n-i}(x)(-1)^{n-i+1}$$
As this holds for each $x\in R$, taking
$$P^{op}_n=\sum_{i=0}^{n-1} P^{op}_i x_{n-i} (-1)^{n-i+1}$$
we have $\sigma^n(x)=P^{op}_n(\lambda^1(x),\ldots,\lambda^n(x))$. Finally, the fact that the degree of $P^{op}_n$ is $n$ is trivially proven inductively, as from the recursive equation we have
$$\deg(P_n^{op}) \le \max_{1\le i<n}\{ \deg(P_i^{op})\}+1$$
Suppose that $\deg(P_i^{op})=i$ for all $i<n$. Then
$$\deg(P_n^{op}) \le \max_{1\le i<n}\{ \deg(P_i^{op})\}+1=n-1+1=n$$
and, moreover,
$$\deg(P^{op}_i x_{n-i})=i+1<n$$
for all $i<n-1$, so
$$\deg(P_n^{op})=\deg(P_{n-1}^{op}x_n)=n$$
\end{proof}

By construction, recall that the Hirzebruch-Newton polynomials $N_n$ provide an algebraic dependency of the Adamas operations in terms of the special $\lambda$-structure $\sigma$.
\begin{equation}
\label{eq:Psigmapsi}
\psi_n(x)=N_n(\sigma^1(x),\ldots,\sigma^n(x))
\end{equation}

The following presentation of Newton's identities
\begin{equation}
\label{eq:NewtonIdentities}
\psi_n(x)= (-1)^{n-1} n \sigma^n(x) -\sum_{i=1}^{n-1} (-1)^{n-i}\sigma^{n-i}(x) \psi_i(x) \quad \quad \forall x\in R
\end{equation}
provides a fast recursive method for computing such degree $n$ Newton polynomials
\begin{align}
\label{eq:recursionsigmapsi}
N_1&=x_1\\
N_n& = (-1)^{n-1} n x_n-\sum_{i=1}^{n-1} (-1)^{n-i} x_{n-i} N_i
\end{align}

If $R$ is a divisible torsion free ring, then it is well known that the process can be reverted and that the special $\lambda$-structure can be expressed as an algebraic combination of Adams operations

\begin{proposition}
\label{prop:psi2sigmaPols}
Let $(R,\sigma)$ be a special $\lambda$-ring such that $R$ is torsion free. Then for every $n>0$, there is a degree $n$ polynomial $L_n(x_1,\ldots,x_n)\in \QQ[x_1,\ldots,x_n]$ such that for each $x\in R$
$$\sigma^n(x)=L_n(\psi_1(x),\ldots,\psi_n(x))$$
satisfying the following recursive relation.
\begin{align*}
L_1&=x_1\\
L_n&=\frac{1}{n} \left( (-1)^{n-1} x_n+ \sum_{i=1}^{n-1} (-1)^{i-1} L_{n-i} x_i\right) \quad \quad \forall n>1
\end{align*}
In this expression, we write ``Y=$\frac{1}{n}$X'' to denote that there exists an element in $R$ which multiplied by $n$ yields $X$ (which is unique by torsion freeness of $R$) and that $Y$ is such element.
\end{proposition}

\begin{proof}
Rearanging Newton's identity \eqref{eq:NewtonIdentities} we have that for each $x\in R$
$$\sigma^n(x) = \frac{1}{n}\left ((-1)^{n-1} \psi_n(x)+ \sum_{i=1}^{n-1} (-1)^{i-1} \sigma^{n-i}(x) \psi_i(x) \right)$$
Moreover, $\sigma^1(x)=x=\psi_1(x)$, so $L_1=x_1$ and, inductively, it is clear that if $\sigma^k(x)=L_k(\psi_1(x),\ldots,\psi_k(x))$ for all $x\in R$ and all $k<n$, then if we define $L_n$ by the given recursion, the following holds for each $x\in R$.
\begin{multline*}
\sigma^n(x) = \frac{1}{n}\left ((-1)^{n-1} \psi_n(x)+ \sum_{i=1}^{n-1} (-1)^{i-1} \sigma^{n-i}(x) \psi_i(x) \right)\\
= \frac{1}{n}\left ((-1)^{n-1} \psi_n(x)+ \sum_{i=1}^{n-1} (-1)^{i-1} L_{n-i}(\psi_1(x),\ldots,\psi_{n-i}(x)) \psi_i(x) \right) = L_n(\psi_1(x),\ldots,\psi_n(x))
\end{multline*}
Finally, the degree $n$ assertion becomes straightforward from the recurrence relation. It is clear by induction that the degree is at most $n$. If we focus on the monomials of the form $x_1^k$ in the polynomials, it is straightforward to prove inductively that for each $n\ge 1$, $\coeff_{x_1^n}(L_n)\ne 0$. We know that $L_1=x_1$ and computing the coefficient of $x_1^n$ in $L_n$ in the recursion equation and taking into account that $\deg(L_k)<n$ for each $k<n$ yields
$$\coeff_{x_1^n} L_n = \frac{1}{n} (-1)^{n-1} \coeff_{x_1^{n-1}} L_{n-1}\ne 0 \quad \quad \forall n>1$$

\end{proof}

Combining the polynomials $L_n$ and $N_n$ with the opposite polynomials $P_k^{op}$ for $k\le n$, we can obtain similar relations between the Adams operations $\psi$ and the non-special $\lambda$-structure on the ring, $\lambda$.

\begin{corollary}
\label{cor:psi2lambdaPols}
Let $(R,\lambda)$ be a torsion free $\lambda$-ring such that the opposite $\lambda$-structure $\sigma$ is special. Then for every $n>0$, there is a degree $n$ polynomial $L_n^{op}(x_1,\ldots,x_n)\in \QQ[x_1,\ldots,x_n]$ such that for each $x\in R$
$$\lambda^n(x)=L_n^{op}(\psi_1(x),\ldots,\psi_n(x))$$
\end{corollary}

\begin{proof}
The existence of the polynomial is clear, as we know that for each $x\in R$
$$\lambda^n(x)=P_n^{op}(\sigma^1(x),\ldots, \sigma^n(x)) = P_n^{op}(L_1(\psi_1(x)),\ldots, L_n(\psi_1(x),\ldots,\psi_n(x)))$$
so
$$L_i^{op}(x_1,\ldots,x_i)=P_n^{op}(L_1(x_1),\ldots, L_n(x_1,\ldots,x_n))\in \QQ[x_1,\ldots,x_n]$$
We can then use the recursive definition of $P_n^{op}$ to verfy inductively that the degree of the resulting polynomial is exactly $n$. For $n=1$ it is trivial, as $L_1^{op}=x_1$. Assume that $\deg(L_i^{op})=i$ for all $i<n$. Applying the recursive equation for $P_n^{op}$ from Proposition \ref{prop:oppositePol} yields
$$L^{op}_n=P_n^{op}(L_1,\ldots,L_n)=\sum_{i=0}^{n-1} P^{op}_i(L_1,\ldots,L_i) L_{n-i} (-1)^{n-i+1} = \sum_{i=0}^{n-1} L^{op}_i L_{n-i} (-1)^{n-i+1}$$
\end{proof}

\begin{corollary}
\label{cor:lambda2psiPols}
Let $(R,\lambda)$ be a $\lambda$-ring such that the opposite structure $\sigma$ is special. Then for every $n>0$, there is a degree $n$ polynomial $N_n^{op}(x_1,\ldots,_n)\in \ZZ[x_1,\ldots,x_n]$ such that for each $x\in R$
$$\psi_n(x)=N_n^{op}(\lambda^1(x),\ldots,\lambda^n(x))$$
\end{corollary}

\begin{proof}
The proof is completely analogous that the one for Corollary \ref{cor:psi2lambdaPols}, taking
$$N_n^{op}=N_n(P_1^{op},\ldots,P_n^{op})$$
\end{proof}

\section{Abstract simplification algorithm}
\label{section:algorithm}

Let us consider the following languages:
\begin{itemize}
\item Let $\SLrings=\{+,-,\cdot,0,1\}$ be the language of (unital) rings.
\item Let $\SLlambda=\SLrings \cup \{\lambda^n,\sigma^n,\psi_n\}_{n\in \NN}$ denote the language of rings with two opposite $\lambda$-structures $\lambda$ and $\sigma$ and Adams operations $\psi$.
\item Let $\SLdrings=\SLrings \cup \{\cdot/n\}_{n\in \NN}$ be the language of divisible rings.
\item Let $\SLdlambda=\SLlambda\cup\{\cdot/n\}_{n\in \NN}$ be the language of divisible rings with a $(\lambda,\sigma,\psi)$-structure.
\end{itemize}
Moreover, let $\ST$ be the $\SLlambda$-theory of $\SLlambda$-structures $(R,+,-,\cdot,0,1,\lambda,\sigma,\psi)$ which satisfy that
\begin{itemize}
\item $(R,+,-,0,1)$ is an abelian unital ring.
\item $\lambda$ and $\sigma$ are oposite $\lambda$-ring structures on $R$
\item $\sigma$ is a special $\lambda$-ring structure.
\item $\psi$ are the Adams operations of $\sigma$
\end{itemize}
Finally, let $\ST_{tf}\supset \ST$ be the theory of $\SLlambda$-structures satisfying $\ST$ such that $(R,+)$ is torsion free.

\begin{definition}
Let $\varphi(x_1,\ldots,x_n)$ be an $\SLlambda$-term in $n$ variables. We define recursively the maximum $\lambda$-depth of the term $\varphi(x_1,\ldots,x_n)$ in the variables $(x_1,\ldots,x_n)$, which we write as $\depth(\varphi)=(d_1,\ldots,d_n)$, as follows.
\begin{enumerate}
\item $$\depth(0)=\depth(1)=(0,\ldots,0)$$
\item $$\depth(x_i)=(0,\ldots,\stackrel{i}{1},\ldots, 0)$$
\item If $\varphi(x_1,\ldots,x_n)$ and $\xi(x_1,\ldots,x_n)$ are $\SLlambda$-terms with respective maximum $\lambda$-depths
$$\depth(\varphi)=(d_1,\ldots,d_n), \quad \quad \depth(\xi)=(d_1',\ldots,d_n')$$
then
$$\depth(-\varphi)=\depth(\varphi)$$
$$\depth(\varphi+\xi)=\depth(\varphi\cdot \xi) =(\max\{d_1,d_1'\},\ldots,\max\{d_n,d_n'\})$$
\item If $\varphi(x_1,\ldots,x_n)$ is a $\SLlambda$-term with $\depth(\varphi)=(d_1,\ldots,d_n)$, then for each $k>0$
$$\depth(\lambda^k(\varphi(x_1,\ldots,x_n)))=\depth(\sigma^k(\varphi(x_1,\ldots,x_n)))=\depth(\psi_k(\varphi(x_1,\ldots,x_n)))=(kd_1,\ldots,kd_n)$$
\end{enumerate}
Analogously, define the maximum $\lambda$-depth of an $\SLdlambda$-term by the relation
$$\depth(\varphi/n)=\depth(\varphi) \quad \forall n\in \NN$$
The following notations will become useful in the next lemmas. We say that $(d_1',\ldots,d_n')\le (d_1,\ldots,d_n)$ if $d_i'\le d_i$ for all $i=1,\ldots,n$ and we say that $(d_1',\ldots,d_n')<(d_1,\ldots,d_n)$ if $(d_1',\ldots,d_n')\le (d_1,\ldots,d_n)$ and $(d_1',\ldots,d_n')\ne (d_1,\ldots,d_n)$.
Finally, given a tuple of positive integers $\overline{d}=(d_1,\ldots,d_n)$, we will write $k[x_{\overline{d}}]$ to denote the ring
$$k[x_{\overline{d}}]:=k[x_{1,1},\ldots,x_{1,d_1},x_{2,1},\ldots,x_{2,d_2},\ldots,x_{n,d_n}]$$
and we will write
$$\lambda^{\overline{d}}(x):=\lambda^{\overline{d}}(x_1,\ldots,x_n):=(\lambda^1(x_1),\ldots,\lambda^{d_1}(x_1),\lambda^1(x_2),\ldots,\lambda^{d_2}(x_2),\ldots,\lambda^{d_n}(x_n))$$
\end{definition}

The idea is that the maximum $\lambda$-depth of an expression identifyies the maximum $\lambda^k$ that has to be computer of each variable $(x_1,\ldots,x_n)$ in order to compute the expression. Here are some examples of depths of terms in $\SLlambda$
\begin{example}
For all polynomials $P(x_1,\ldots,x_n)$ in $\SLrings$ involving all variables $x_1,\ldots,x_n$
$$\depth(P(x_1,\ldots,x_n))=(1,\ldots,1)$$
\end{example}

\begin{example}
For all $k>0$
$$\depth(\lambda^k(x))=\depth(\sigma^k(x))=\depth(\psi_k(x))=k \quad \forall k$$
\end{example}

\begin{example}
$$\depth(\lambda^3(y+\sigma^2(xy)+\psi_3(y))\lambda^2(x)) = (6,9)$$
\end{example}

\begin{remark}
If $\overline{d'}\le \overline{d}$, then we have a canonical injection
$$k[x_{\overline{d'}}]\hookrightarrow k[x_{\overline{d}}]$$
\end{remark}

First of all, observe that the algebraic relations from the previous section allow us to rewrite any expression in a $\lambda$-ring $(R,\lambda)$ whose opposite $\lambda$-structure $\sigma$ is special exclusively in terms of either $\lambda$ or $\sigma$, without changing the maximum $\lambda$-depth of such expression. More explicitly, the following Lemma holds.

\begin{lemma}
Let $\varphi(x_1,\ldots,x_n)$ be a $\SLlambda$-term. Then there exists a $\SLrings\cup\{\sigma^n\}$-term $\varphi_\sigma$ and a $\SLrings\cup\{\lambda^n\}$-term $\varphi_\lambda$ with the same maximum $\lambda$-depth as $\varphi$ such that
$$\ST \models \forall x_1, \ldots, x_n\, \varphi(x_1,\ldots,x_n)=\varphi_\sigma(x_1,\ldots,x_n)$$
$$\ST \models \forall x_1, \ldots, x_n\, \varphi(x_1,\ldots,x_n)=\varphi_\lambda(x_1,\ldots,x_n)$$
\end{lemma}

\begin{proof}
Let us prove it inductively in $\lambda$-terms. If $\varphi$ is a constant or a variable, then $\varphi_\sigma=\varphi_\lambda=\varphi$ stisfies the result. If $\varphi=\varphi'+\varphi''$, then
$$\varphi_\sigma=\varphi'_\sigma+\varphi''_\sigma$$
satisfies the statement, as in both cases
$$\depth(\varphi_\sigma)=\max(\depth(\varphi'_\sigma),\depth(\varphi''_\sigma))=\max(\depth(\varphi'),\depth(\varphi''))=\depth(\varphi)$$
Similarly, if $\varphi=\varphi'\cdot \varphi''$, then, analogously, $\varphi_\sigma=\varphi_\sigma'\cdot\varphi'_\sigma$. If $\varphi=-\varphi'$, then $\varphi_\sigma=-\varphi_\sigma$ holds.
If $\varphi=\sigma^k(\varphi')$ for some $k$, then $\varphi_\sigma=\sigma^k(\varphi'_\sigma)$ clearly satisfies the result, as
$$\depth(\varphi_\sigma)=k\depth(\varphi'_\sigma)=k\depth(\varphi')=\depth(\varphi)$$
If $\varphi=\lambda^k(\varphi')$ for some $k$, then take
$$\varphi_\sigma=P_N^{op}(\sigma^1(\varphi'_\sigma),\ldots,\sigma^k(\varphi'_\sigma))$$
Taking into account the induction hypotyesis on $\varphi'$, by Proposition \ref{prop:oppositePol}, $\ST \models \forall x_1,\ldots,x_n\,  \varphi_\sigma(x_1,\ldots,x_n)=\varphi(x_1,\ldots,x_n)$ and we have
$$\depth(\varphi_\sigma)=\max(\depth(\sigma^1(\varphi'_\sigma),\ldots,\depth(\sigma^k(\varphi'_\sigma))))=\max(\depth(\varphi'),\ldots,k\depth(\varphi'))=k\depth(\varphi')=\depth(\varphi)$$
Analogously, if $\varphi=\psi_k(\varphi')$ for some $k$, then take
$$\varphi_\sigma=N_k(\sigma^1(\varphi'_\sigma),\ldots,\sigma^k(\varphi'_\sigma))$$
The previous computation of the depth of $\varphi_\sigma$ also works in this case and, by construction of $\psi_k$ (Definition \ref{def:AdamsOperations}), $\ST \models \forall x_1,\ldots,x_n\,  \varphi_\sigma(x_1,\ldots,x_n)=\varphi(x_1,\ldots,x_n)$.
The result for $\varphi_\lambda$ is completely analogous, setting $N_k^{op}$ from Corollary \ref{cor:lambda2psiPols} instead of $N_k$ from Definition \ref{def:AdamsOperations}.
\end{proof}

\begin{theorem}
\label{thm:exprSimplification}
Let $\varphi(x_1,\ldots,x_n)$ be a $\SLlambda$-term with $\lambda$-depth $\depth(\varphi)=(d_1,\ldots,d_n)$. Then there exist unique polynomials
$$P_\varphi^\lambda(x_{1,1},\ldots,x_{1,d_1},x_{2,1},\ldots,x_{n,d_n})\in \ZZ[x_{1,1},\ldots,x_{n,d_n}]=\ZZ[x_{\depth(\varphi)}]$$
$$P_\varphi^\sigma(x_{1,1},\ldots,x_{1,d_1},x_{2,1},\ldots,x_{n,d_n})\in \ZZ[x_{1,1},\ldots,x_{n,d_n}]=\ZZ[x_{\depth(\varphi)}]$$
such that
$$\ST \models \forall x_1,\ldots,x_n \, \varphi(x_1,\ldots,x_n)=P_\varphi^\lambda(\lambda^1(x_1),\dots,\lambda^{d_1}(x_1),\lambda^1(x_2),\ldots, \lambda^{d_n}(x_n))$$
$$\ST \models \forall x_1,\ldots,x_n \, \varphi(x_1,\ldots,x_n)=P_\varphi^\sigma(\sigma^1(x_1),\dots,\sigma^{d_1}(x_1),\sigma^1(x_2),\ldots, \sigma^{d_n}(x_n))$$
\end{theorem}

\begin{proof}
We will address first existence the polynomials. The existence of $P_\varphi^\lambda$ clearly follows from the existence of $P_\varphi^\sigma$, as Proposition \ref{prop:oppositePol} implies that
$$\ST\models \forall x_1,\ldots,x_n \, P_\varphi^\sigma(\sigma^1(x_1),\dots,\sigma^{d_1}(x_1),\sigma^1(x_2),\ldots, \sigma^{d_n}(x_n)) = P_\varphi^\sigma(P_1^{op}(\lambda^1(x_1)),\ldots,P_{d_1}^{op}(\lambda^1(x_1),\ldots,\lambda^n(x_1)),\ldots,P_{d_n}^{op}(\lambda^1(x_n),\ldots,\lambda^{d_n}(x_n))$$
so we can take
$$P_\varphi^\lambda(x_{1,1},\ldots,x_{n,d_n})=P_\varphi^\sigma(P_1^{op}(x_1),\ldots,P_{d_1}^{op}(x_{1,1},\ldots,x_{1,d_1}),\ldots, P_{d_n}^{op}(x_{n,1}\ldots,x_{n,d_n}))$$

Let us focus on the existence of $P_\varphi^\sigma$. By the previous lemma we can assume without loss of generality that $\lambda^k$ and $\psi_k$ do not appear in $\varphi$. Let us start by proving that if $\varphi$ has the form $\varphi=\sigma^k(\varphi')$ where $k\ge 0$ and $\varphi'$ is a $\SLrings$-term, then $\varphi$ satisfies the Theorem. Let us prove it inductively on terms of $\varphi'$.

If $\varphi'=0$ or $\varphi'=1$, then $P=0$ or $P=1$ satisfy the Theorem respectively. If $\varphi'=x_j$ for some $j$, then
$$\depth(\sigma^k(\varphi'))=(0,\ldots,\stackrel{j}{k},\ldots, 0)$$
and $P_{\sigma^k(x_j)}^\sigma=x_{j,k}$, considered as a polynomial in $\ZZ[x_{1,1},\ldots,x_{j,1},\ldots,x_{j,k},\ldots, x_{n,1}]$ clearly satisfies the result for each $k$.

If $\varphi'=\varphi'_1+\varphi'_2$, then
$$\ST\models \forall x_1,\ldots,x_n \, \sigma^k(\varphi'_1+\varphi'_2)= \sum_{j=0}^k \sigma^j(\varphi'_1)\sigma^{k-j}(\varphi'_2)$$
By induction hypothesis, then for each $j$ and $l=1,2$
$$\ST\models \forall x_1,\ldots,x_n  \, \sigma^j(\varphi'_l(x_1,\ldots,x_n))=P_{\sigma^j(\varphi'_l)}^\sigma(\sigma^{j\depth(\varphi'_l)}(x))$$
where $P_{\sigma^j(\varphi'_l)}^\sigma \in \ZZ[x_{k\depth(\varphi'_l)}]$. Thus
$$\ST\models \forall x_1,\ldots,x_n \, \sigma^k(\varphi'_1(x_1,\ldots,x_n)+\varphi'_2(x_1,\ldots,x_n))= \sum_{j=0}^k P_{\sigma^j(\varphi'_1)}^\sigma(\sigma^{j\depth(\varphi'_1)}(x))P_{\sigma^{k-j}(\varphi'_2)}^\sigma(\sigma^{(k-j)\depth(\varphi'_2)}(x))$$
and a straightforward computation shows that last polynomial lives in
$$\ZZ[x_{k\depth(\varphi_1'+\varphi_2')}]=\ZZ[x_{\depth(\varphi)}]$$
In an analogous way, if $\varphi'=\varphi_1' \cdot \varphi_2'$, then
\begin{multline*}
\ST\models \forall x_1,\ldots,x_n \, \sigma^k(\varphi'_1(x_1,\ldots,x_n)\cdot \varphi'_2(x_1,\ldots,x_n))\\
= P_k(\sigma^1(\varphi'_1(x_1,\ldots,x_n)),\ldots,\sigma^k(\varphi'_1(x_1,\ldots,x_n)),\sigma^1(\varphi'_2(x_1,\ldots,x_n)),\ldots, \sigma^k(\varphi'_2(x_1,\ldots,x_n)))\\
 = P_k(P_{\sigma^1(\varphi'_1)}^\sigma(\sigma^{\depth(\varphi'_1)}(x)),\ldots,P_{\sigma^k(\varphi'_1)}^\sigma(\sigma^{k\depth(\varphi'_1)}(x)),P_{\sigma^1(\varphi'_2)}^\sigma(\sigma^{\depth(\varphi'_2)}(x)),\ldots,P_{\sigma^k(\varphi'_2)}^\sigma(\sigma^{k\depth(\varphi'_2)}(x)))
\end{multline*}
composing the last polynomials $P_k(P_{\sigma^1(\varphi'_1)}^\sigma,\ldots,P_{\sigma^k(\varphi'_2)}^\sigma)$ we obtain the desired polynomial, which clearly lives in $\ZZ[x_{\max\{k\depth(\varphi'_1),k\depth(\varphi'_2)\}}]=\ZZ[x_{\depth(\varphi)}]$.
Finally, if $\varphi=\sigma^k(-\varphi')$, then we can use the identity \eqref{eq:sigma-x} and Proposition \ref{prop:oppositePol} to obtain that
\begin{multline*}
\ST\models \forall x_1,\ldots,x_n \, \sigma^k(-\varphi'(x_1,\ldots,x_n))= (-1)^k \lambda^k(\varphi'(x_1,\ldots,x_n))\\
=(-1)^kP_k^{op}(\sigma^1(\varphi'(x_1,\ldots,x_n)),\ldots,\sigma^k(\varphi'(x_1,\ldots,x_n)))
\end{multline*}
And we can then repeat the previous argument on $\sigma^j(\varphi')$ for each $j=1,\ldots,n$, obtaining a polynomial
$$P_k^{op}(P_{\sigma^1(\varphi')}^\sigma(\lambda^{\depth(\varphi')}(x)),\ldots,P_{\sigma^k(\varphi')}^\sigma(\lambda^{\depth(\varphi')}(x)))\in \ZZ[x_{k\depth(\varphi')}]=\ZZ[x_{\depth(\sigma^k(-\varphi'))}]$$

Now, we are ready to prove the main theorem inductively on terms of $\SLlambda$-expressions. We have already seen that the result holds if $\varphi$ is a constant or a variable. Suppose that the result holds for two expressions $\varphi_1$ and $\varphi_2$. Then clearly
$$P_{\varphi_1+\varphi_2}^\sigma=P_{\varphi_1}^\sigma+P_{\varphi_2}^\sigma\in \ZZ[x_{\max\{\depth(\varphi_1),\depth(\varphi_2)\}}]=\ZZ[x_{\depth(\varphi_1+\varphi_2)}]$$
$$P_{\varphi_1\cdot\varphi_2}^\sigma=P_{\varphi_1}^\sigma\cdot P_{\varphi_2}^\sigma\in \ZZ[x_{\max\{\depth(\varphi_1),\depth(\varphi_2)\}}]=\ZZ[x_{\depth(\varphi_1\cdot\varphi_2)}]$$
$$P_{-\varphi}^\sigma=-P_\varphi^\sigma\in \ZZ[x_{\depth(\varphi)}]= \ZZ[x_{\depth(-\varphi)}]$$
and, for each $k$, we have
$$\ST \models \forall x_1,\ldots,x_n \, \sigma^k(\varphi_1(x_1,\ldots,x_n))=  \sigma^k(P_{\varphi_1}^\sigma(\sigma^{\depth(\varphi_1)}(x)))$$
so, applying the previous argument, we know that there exists $P_{\sigma^k(P_{\varphi_1}(x_{\depth(\varphi_1)}))}^\sigma$ such that
$$\ST \models \forall x_1,\ldots,x_n\, \sigma^k(P_{\varphi_1}^\sigma(\sigma^{\depth(\varphi_1)}(x))) = P_{\sigma^k(P_{\varphi_1}(x_{\depth(\varphi_1)}))}^\sigma(\sigma^1(\sigma^1(x_1)),\ldots,\sigma^k(\sigma^1(x_1)),\ldots,\sigma^k(\sigma^{d_n}(x_n)))$$
where $(d_1,\ldots,d_n)=\depth(\varphi_1)$. Then, we can simply use the fact that $\sigma$ is special to get that for each $a,b,c$
$$\sigma^a(\sigma^b(x_c))=P_{a,b}(\sigma^1(x_c),\ldots,\sigma^{ab}(x_c))$$
To obtain that
$$P_{\sigma^k(\varphi_1)}^\sigma=P_{\sigma^k(P_{\varphi_1}(x_{\depth(\varphi_1)}))}^\sigma(P_{1,1},P_{2,1},\ldots,P_{k,1},P_{1,2},\ldots,P_{k,2},\ldots,P_{k,d_k})\in \ZZ[x_{k\depth(\varphi_1)}]$$

Finally, let us prove uniqueness of the polynomials. Suppose that given $\varphi(x_1,\ldots,x_n)$ of depth $(d_1,\ldots,d_n)$ there exist $P_\varphi^\sigma$ and $Q_\varphi^\sigma$ with the given properties. In particular, if we apply the result to the free special $\lambda$-ring in variables $x_1,\ldots,x_n$ we have that 
$$P_\varphi^\sigma(\sigma^1(x_1),\ldots,\sigma^{d_n}(x_n))-Q_\varphi^\sigma(\sigma^1(x_1),\ldots,\sigma^{d_n}(x_n))=0$$
but in the free special $\lambda$-ring generated by such variables, the elements $\sigma^j(x_i)$ are all algebraically independent, so $P_\varphi^\sigma-Q_\varphi^\sigma=0$. The proof for uniqueness of $P_\varphi^\lambda$ is completely analogous.
\end{proof}

Then, we propose the following algorithm for computing the polynomial $P_\varphi^\lambda$

\begin{algorithm}[H]
\caption{Simplification algorithm for $\SLlambda$-expressions}
\label{alg:simpAlg}
\begin{algorithmic}[1]
\Procedure{$\lambda$-simp}{$\varphi$,$x_1$,\ldots,$x_n$}
\State Compute $(d_1,\ldots,d_n)=\depth(\varphi)$
\For{each $x_i$}
	\For{each $k\le d_i$}
		\State Compute recursively $\Psi_{i,k}=P_{\psi_k(x_i)}^\lambda=N_k^{op}(x_{i,1},\ldots,x_{i,k})$
	\EndFor
	\For{each $a,b$, $ab\le d_i$}
		\State Compute $\mu_{a,b}(x_i)=P_{\psi_a(\lambda^b(x_i))}^\lambda=L_b^{op}(\Psi_{i,a+1},\Psi_{i,a+2},\ldots,\Psi_{i,a+b})$
	\EndFor
\EndFor
\State \textbf{Scan} $\varphi$ \textbf{and substitute}
$$\psi_k(\omega) \rightarrow \text{\texttt{ApplyPsi}}(\omega,k,\{\mu_{a,b}(x_i)\})$$
$$\lambda^k(\omega) \rightarrow \left \{ \begin{array}{ll}
L_k^{op}(\text{\texttt{ApplyPsi}}(\omega,1,\{\mu_{a,b}(x_i)\}),\ldots,\text{\texttt{ApplyPsi}}(\omega,k,\{\mu_{a,b}(x_i)\})) & \text{if }\omega \ne x_1,\ldots,x_n\\
x_{i,k} & \text{if }\omega=x_i
\end{array}\right.$$
$$\sigma^k(\omega) \rightarrow L_k(\text{\texttt{ApplyPsi}}(\omega,1,\{\mu_{a,b}(x_i)\}),\ldots,\text{\texttt{ApplyPsi}}(\omega,k,\{\mu_{a,b}(x_i)\})$$
\textbf{to get} $\hat{\varphi}$.
\Comment Polynomials $L_k$, $L_k^{op}$, $N_k^{op}$ are pre-computed using recurrences from Propositions \ref{prop:oppositePol}, \ref{prop:psi2sigmaPols} and formulas from Corollaries \ref{cor:psi2lambdaPols} and \ref{cor:lambda2psiPols}.
\State \textbf{Execute} $\hat{\varphi}$ and \textbf{expand the resulting rational polynomial}
\EndProcedure
\Function{ApplyPsi}{$\varphi$,$k$,$\{\mu_{a,b}(x_i)\}$}
\For{each $x_i$}
	\For{each $j$}
		\State Substitute $x_{i,j}$ by $\mu_{k,j}(x_i)$ in $\varphi$
	\EndFor
\EndFor
\EndFunction
\end{algorithmic}
\end{algorithm}

\begin{theorem}
\label{thm:mainAlgorithm}
For each $\SLlambda$-expression  $\varphi$ in a torsion free $\lambda$-ring $R$, Algorithm \ref{alg:simpAlg} computes the polynomial $P_\varphi^\lambda$ from Theorem \ref{thm:exprSimplification}.
\end{theorem}

\begin{proof}
First of all, observe that, due to Corollary \ref{cor:psi2lambdaPols}
$$\ST_{tf} \models \forall x \, \psi_a(\lambda^b(x))=\psi_a(L_b^{op}(\psi_1(x),\ldots,\psi^b(x))) = L_b^{op}(\psi_a(\psi_1(x)),\ldots,\psi_a(\psi_b(x)))=L_b^{op}(\psi_{a+1}(x),\ldots,\psi_{a+b}(x))$$
Therefore, it is clear that $\mu_{a,b}(x_i)$ contain $P_{\psi_a(\lambda^b(x_i))}^\lambda$, as asserted.

Given any $\SLlambda$-expression $\varphi$ with $\depth(\varphi)=(d_1,\ldots,d_n)$, by Theorem \ref{thm:exprSimplification} we have that for each integer $k$
$$\ST \models \forall x_1,\ldots, x_n \, \psi_k(\varphi(x_1,\ldots,x_n))=\psi_k(P_\varphi^\lambda(\lambda^1(x_1),\ldots, \lambda^{d_n}(x_n)))$$
On the other hand, as $\psi_k$ are homomorphisms, we have
$$\ST \models \forall x_1,\ldots, x_n \, \psi_k(P_\varphi^\lambda(\lambda^1(x_1),\ldots, \lambda^{d_n}(x_n))) = P_\varphi^\lambda(\psi_k(\lambda^1(x_1)),\ldots, \psi_k(\lambda^{d_n}(x_n)))$$
Thus, applying $\psi_k$ to any $\SLlambda$-expression $\varphi$ is equivalent to substituting $\lambda^j(x_i)$ (a.k.a., variable $x_{i,j}$ by the polynomial $\mu_{n,j}(x_i)$ in the associated polynomial $P_\varphi^\lambda$.

On the other hand, Proposition \ref{prop:psi2sigmaPols} and Corollary \ref{cor:psi2lambdaPols}, we obtain that
$$\ST_{tf}\models \forall x_1,\ldots, x_n \, \lambda^k(\varphi(x_1,\ldots,x_n)) = L_n^{op}(\psi_1(\varphi(x_1,\ldots,x_n)),\ldots,\psi_k(\varphi(x_1,\ldots,x_n)))$$
$$\ST_{tf}\models \forall x_1,\ldots, x_n \,  \sigma^k(\varphi(x_1,\ldots,x_n)) = L_n(\psi_1(\varphi(x_1,\ldots,x_n)),\ldots,\psi_k(\varphi(x_1,\ldots,x_n)))$$

Thus, if $\omega=P_\varphi^\lambda(\lambda^{\depth(\varphi)}(x))$ for some $\varphi$ then applying the function \texttt{ApplyPsi}($\omega$,$k$,$\{\mu_{a,b}(x_i)\}$) yields an integer polynomial $Q_{\psi^k(\omega)}\in \ZZ[x_{k\depth(\omega)}]=\ZZ[x_{k\depth(\varphi)}]$ such that
$$\ST \models \forall x_1,\ldots, x_n \, Q_{\psi^k(\omega)}(\lambda^{k\depth(\varphi)}(x))=\psi^k(P_\varphi(\lambda^{\depth(\varphi)}(x)))=\psi^k(\varphi(x_1,\ldots,x_n))=P_{\psi^k(\varphi)}^\lambda(\lambda^{k\depth(\varphi)}(x))$$
By uniqueness of the polynomial $P_{\psi^k(\varphi)}^\lambda$ of Theorem \ref{thm:exprSimplification}, we obtain that $Q_{\psi^k(\omega)}=P_{\psi_k(\varphi)}^\lambda$.

Similarly, under the same assumptions on $\omega$, if now $\omega'$ is either $\psi_k(\omega)$, $\lambda^k(\omega)$ or $\sigma^k(\omega)$, then applying the substitutions and executing \texttt{ApplyPsi} in the main loop of the algorithm to such $\omega$ yields a rational polynomial $Q_{\omega'}$ on variables $x_{i,j}$ such that
$$\ST_{tf} \models \forall x_1,\ldots, x_n \, Q_{\omega'}(\lambda^{k\depth(\varphi)}(x)) = \omega'(x_1,\ldots,x_n) = P_{\omega'}^\lambda(\lambda^{k\depth(\varphi)}(x))$$
Multiplying by a big enough natural number $N$, we would get that $NQ_{\omega'}$ is an integral polynomial such that
$$\ST_{tf} \models \forall x_1,\ldots, x_n \, NQ_{\omega'}(\lambda^{k\depth(\varphi)}(x)) = NP_{\omega'}^\lambda(\lambda^{k\depth(\varphi)}(x))=N\omega'(x_1,\ldots,x_n)=P_{N\omega'}^\lambda(\lambda^{k\depth(\varphi)}(x))$$
By uniqueness of $P_{N\omega'}^\lambda$ of Theorem \ref{thm:exprSimplification}, we obtain that $NQ_{\omega'}=P_{N\omega'}^\lambda=NP_{\omega'}^\lambda$. Assuming that the rings are torsion free, this implies that $Q_{\omega'}=P_{\omega'}^\lambda$.

On the other hand, by construction, each expression of the form $\lambda^j(x_i)$ is transformed at the main loop into a variable $x_{i,j}$ which is exactly $P_{\lambda^j(x_i)}^\lambda$. As the substitutions of the main loop are executed from the inner-most parts of the formula $\varphi$, the previous argument then shows that after each set of substitutions is made, the resulting formula is an integral polynomial which coincides with the simplified polynomial predicted by Theorem \ref{thm:exprSimplification} for the corresponding expression. Therefore, when the algorithm ends, the result is $P_\varphi^\lambda$.
\end{proof}

\section{Simplification of expressions in the Grothendieck ring of Chow motives}
\label{section:KChow}

Let $\KK$ be a field of characteristic $0$ and let $K_0(\ChowK)$ denote the Grothendieck ring of motives over $k$ with rational coefficients. By \cite{H07}, $K_0(\ChowK)$ admits two natural opposite $\lambda$-structures, induced by the symmetric product of varieties. Given a Chow Motive $M$, take
$$\lambda^n([M]) = [\Sym^n(M)]$$
More precisely, define the symmetric product $\Sym^n(M)$ as the image of $\frac{1}{n!} \sum_{\alpha\in S_n} \alpha : X^{\otimes n}\to X^{\otimes n}$. Similarly, take
$$\sigma^n([M])=[\Alt^n(M)]$$
where $\Alt^n$ is the alternating product, defined as the image of $\frac{1}{n!} \sum_{\alpha \in S_n} (-1)^{\op{sg}{\alpha}} \alpha: X^{\otimes n} \to X^{\otimes n}$.
Heinloth \cite{H07} proves that $\lambda^n$ and $\sigma^n$ are $\lambda$-structures $K_0(\ChowK)$, but only $\sigma^n$ is a special $\lambda$-structure.

Let $\LL=[\mathbb{A}^1]\in K_0(\ChowK)$ be the Lefschetz object. In addition to $K_0(\ChowK)$, we will also consider its localization $K_0(\ChowK)[\LL^{-1}]$ and its dimensional completion
$$ \hat{K}_0 (\ChowK)= \left\{ \sum_{r\geq 0} [Y_r] \, \LL^{-r} \, \middle | \, [Y_r]\in K_0(\ChowK) \text{ with }\dim Y_r -r \longrightarrow -\infty \right\}.$$
as, for some applications, we will need elements of the form $\LL^n$ or $\LL^n-1$ to be invertible. Both $\lambda$-structures then extend to the localization and its completion, as $\LL$ is $1$-dimensional for the opposite structure $\sigma$ (and, as such, $\psi_n(\LL)=\LL^n$ for all $n$). 

Given a motive $M$, the formal sum
$$\sum_{n\ge 0} [\Sym^n M] t^n = \lambda_t([M])$$
is usually called the motivic zeta function of $M$ and it is classically denoted by $Z_M(t)$. The motives and zeta functions of some varieties then admit some interesting decompositions and closed expressions when expressed in terms of the $\lambda$-ring structure $\lambda^n=\Sym^n$. For instance, if $X$ is a projective curve of genus $g\ge 2$, then by \cite{Kapranov00} its motive splits as $[X]=1+[h^1(X)]+\LL$ and, moreover, the zeta function of $[h^1(X)]$ ia a polynomial of degree $2g$, usually called $P_X(t)$. Then
$$Z_X(t)=\lambda_t([X]) =\lambda_t(1+[h^1(X)]+\LL)=\lambda_t(1)\lambda_t([h^1(X)])\lambda_t(\LL)=\frac{1}{1-t} P_X(t) \frac{1}{1-\LL t}=\frac{P_X(t)}{(1-t)(1-\LL t)}$$
The motive $[h_1(X)]$ and its zeta function/polynomial appear in many additional computations of motives of geometrical structures associated to the curve $X$. For instance the motive of the Jacobian of $X$ is
$$[\Jac(X)]=\sum_{i=0}^{2g}\lambda^i([h^1(X)])=P_X(1)$$
On the other hand, the elements $\lambda^i([h^1(X)])$ have certain relations among them, which can be deduced from the following functional equation due to Kapranov \cite{Kapranov00} (see also \cite{H07}).
$$Z_X\left(\frac{1}{\LL t}\right) = \LL^{1-g} t^{2-2g} Z_X(t)$$
In particular, the following properties hold
\begin{enumerate}
\item For all $i>2g$, $\lambda^i([h^1(X)])=0$.
\item For all $g<i\le 2g$,
$$\lambda^i([h^1(X)]) = \LL^{i-g} \lambda^{2g-i}([h^1(X)])$$
\end{enumerate}
Let us write
$$a_i(X)=\lambda^i([h^1(X)]) \quad \quad i=1,\ldots,g$$
if $X$ is clear from the context, we will omit it and simply write $a_i=a_i(X)$.Then
$$P_X(t)=1+a_1(X)t+a_2(X)t^2+\ldots+a_g(X)t^g + \LL a_{g-1} t^{g+1}+\ldots +\LL^{g-1} a_1 t^{2g-1}+ \LL^g t^{2g}$$
It is therefore straightforward to see that any algebraic expression in the $\lambda$-terms $\lambda^k([X])$ can be expanded as a polynomial in $a_1,a_2=\lambda^2(a_1),\ldots, a_g=\lambda^g(a_1)$ and $\LL$. There is, nevertheless, a key difference in working with expressions involving $\lambda$-operations on $[X]$ and on $a_1(X)$. The dimension of $[X]$ in $(\hat{K}_0(\ChowK),\lambda)$ is infinite, whereas the dimension of $a_1(X)$ is finite ($2g$). Thus, expressions involving arbitrary large numbers of $\lambda$ operations on $X$ are always equivalent to polynomials in the same fixed number of variables $a_1,\ldots,a_g$ and $\LL$. Using Theorem \ref{thm:exprSimplification}, we can extend this intuitive argument to further arbitrary expressions spanned by curves in $\hat{K}_0(\ChowK)$.

\begin{theorem}
\label{thm:exprSimpMotives}
Let $X_1,\ldots,X_n$ be smooth projective algebraic curves over a field $\KK$ of characteristic $0$ of respective genra $g_1,\ldots,g_n$. Let $\varphi(\LL,[X_1],\ldots,[X_n])$ be any $\SLlambda$-term in the sub-$(\lambda,\sigma,\psi)$-ring of $\hat{K}_0(\ChowK)$ spanned by the curves and the Lefschetz object. Then there exists an integral polynomial $P_\varphi\in \ZZ[L,a_{1,1},\ldots,a_{1,g_1},\ldots,a_{n,g_n}]$ such that
$$\varphi(\LL,[X_1],\ldots,[X_n])=P_\varphi(\LL,a_1(X_1),\ldots,a_{g_1}(X_1),\ldots,a_{g_n}(X_n))$$
\end{theorem}

\begin{proof}
Expressing $[X_i]=1+a_1(X)+\LL$, we can rewrite $\varphi$ as some expression $\varphi'$ in $a_1(X_1),\ldots,a_1(X_n)$
$$\varphi(\LL,[X_1],\ldots,[X_n])=\varphi(\LL,1+a_1(X_1)+\LL,\ldots,1+a_1(X_n)+\LL)=\varphi'(\LL,a_1(X_1),\ldots,a_1(X_n))$$
Let $(d_0,d_1,\ldots,d_n)=\depth(\varphi')$. Then Theorem \ref{thm:exprSimplification} implies that there exists a polynomial
$$P_{\varphi'}^\lambda\in \ZZ[x_{(d_0,\ldots,d_n)}]$$
such that
\begin{equation}
\label{eq:exprSimpMotives1}
\varphi'(\LL,a_1(X_1),\ldots,a_1(X_n))=P_{\varphi'}^\lambda(\lambda^1(\LL),\ldots,\lambda^{d_0}(\LL),\lambda^1(a_1(X)),\ldots, \lambda^{d_n}(a_1(X_n)))
\end{equation}
Nevertheless, we know the following relations.
\begin{itemize}
\item $\lambda^k(\LL)=\LL^k \quad \forall k$
\item $\lambda^k(a_1(x_i))=a_k \quad \forall 1\le k\le g_i$
\item $\lambda^k(a_1(X_i))=\LL^{k-g} a_{2g-k}(X_i) \quad \forall g_i<k\le 2g_i$
\item $\lambda^k(a_1(X_i))=0 \quad \forall k>2g_i$
\end{itemize}
Thus, all the generators in the right hand side of the equality \eqref{eq:exprSimpMotives1} are always algebraically generated by $\LL$ and $a_k(X_i)$ for $1\le k\le g_i$. Substituting the corresponding algebraic relations at the right hand side of \eqref{eq:exprSimpMotives1}, we obtain the desired polynomial $P_\varphi$.
\end{proof}

Finally, with some little adaptations, Algorithm \ref{alg:simpAlg} can be used to compute the polynomial $P_\varphi$ from the previous Theorem.

\begin{algorithm}[H]
\caption{Simplification algorithm for motivic expressions}
\label{alg:motiveSimpAlg}
\begin{algorithmic}[1]
\Procedure{motive-simp}{$\varphi$,$X_1$,\ldots,$X_n$}
\State Rewrite $\varphi(\LL,[X_1],\ldots,[X_n])=\varphi'(\LL,a_1(X_1),\ldots,a_1(X_n))$
\State Compute $(d_0,d_1,\ldots,d_n)=\depth(\varphi')$
\For{each curve $X_i$}
	\For{each $k\le d_i$}
		\State Compute recursively $\Psi_{i,k}=P_{\psi_k(a_1(X_i))}^\lambda=N_k^{op}(a_{i,1},\ldots,a_{i,g_i},La_{i,g_i-1},\ldots, L^{g_i-1}a_{i,1}, L^{g_i},0,
\ldots,0)$
		\State $\mu_{\alpha,1}(X_i)=\Psi_{i,\alpha}$
	\EndFor
	\For{each $\alpha,\beta$, $\alpha\beta \le d_i$, $2\le \beta\le g_i$}
		\State Compute $\mu_{\alpha,\beta}(X_i)=P_{\psi_\alpha(a_\beta(X_i)))}^\lambda=L_\beta^{op}(\Psi_{i,\alpha+1},\Psi_{i,\alpha+2},\ldots,\Psi_{i,\alpha+\beta})$
	\EndFor
\EndFor
\State \textbf{Scan} $\varphi'$ \textbf{and substitute}
$$\psi_k(\omega) \rightarrow \text{\texttt{ApplyPsi}}(\omega,k,\{\mu_{\alpha,\beta}(X_i)\})$$
$$\lambda^k(\omega) \rightarrow \left \{ \begin{array}{ll}
L_k^{op}(\text{\texttt{ApplyPsi}}(\omega,1,\{\mu_{\alpha,\beta}(X_i)\}),\ldots,\text{\texttt{ApplyPsi}}(\omega,k,\{\mu_{\alpha,\beta}(X_i)\})) & \text{if }\omega \ne a_1(X_1),\ldots,a_1(X_n)\\
a_{i,k} & \text{if }\omega=a_1(X_i), \quad k\le g\\
L^{k-g_i} a_{i,2g_i-k} & \text{if }\omega=a_1(X_i), \quad g_i<k< 2g_i\\
L^{g_i} & \text{if }\omega=a_1(X_i), \quad k=2g_i\\
0 & \text{if }\omega=a_1(X_i), \quad k>2g_i
\end{array}\right.$$
$$\sigma^k(\omega) \rightarrow L_k(\text{\texttt{ApplyPsi}}(\omega,1,\{\mu_{\alpha,\beta}(X_i)\}),\ldots,\text{\texttt{ApplyPsi}}(\omega,k,\{\mu_{\alpha,\beta}(X_i)\})$$
\textbf{to get} $\hat{\varphi}$.
\Comment Polynomials $L_k$, $L_k^{op}$, $N_k^{op}$ are pre-computed using recurrences from Propositions \ref{prop:oppositePol}, \ref{prop:psi2sigmaPols} and formulas from Corollaries \ref{cor:psi2lambdaPols} and \ref{cor:lambda2psiPols}.
\State \textbf{Execute} $\hat{\varphi}$ and \textbf{expand the resulting rational polynomial}
\EndProcedure
\Function{ApplyPsi}{$\varphi$,$k$,$\{\mu_{\alpha,\beta}(x_i)\}$}
\State Substitute $L$ by $L^k$ in $\varphi$
\For{each curve $X_i$}
	\For{$j=1,\ldots,g_i$}
		\State Substitute $a_{i,j}$ by $\mu_{k,j}(X_i)$ in $\varphi$
	\EndFor
\EndFor
\EndFunction
\end{algorithmic}
\end{algorithm}

\begin{theorem}
\label{thm:mainAlgirhtmMotives}
For each $\SLlambda$-term $\varphi(\LL,[X_1],\ldots,[X_n])$ spanned by motves of curves $X_1,\ldots,X_n$ and the Lefschetz object $\LL$, Algorithm \ref{alg:motiveSimpAlg} computes the polynomial $P_\varphi$ from Theorem \ref{thm:exprSimpMotives}.
\end{theorem}

\begin{proof}
It follows directly from Theorem \ref{thm:mainAlgorithm}, using $a_{i,1}=a_1(X_i)$ as generators $x_i$ of $\varphi'$, once we take into account that variables $a_{i,k}$ with $k>g_i$ are not algebraically free anymore, but rather given in terms of $a_{i,1},\ldots,a_{i,g_i}$ and $\LL$ as
$$\left\{ \begin{array}{ll}
L^{k-g_i} a_{i,2g_i-k} & \text{if }g_i<k< 2g_i\\
L^{g_i} & \text{if }k=2g_i\\
0 & \text{if }k>2g_i
\end{array}\right.$$
\end{proof}

\section{Computational verification of a conjectural formula for the motives of moduli spaces of twisted Higgs bundles and Lie algebroid connections}
\label{section:Higgs}

Let $X$ be a smooth complex projective curve of genus $g\ge 2$. Let $L$ be line bundle on $X$ of degree $d_L=\deg(\SL)$ and suppose that $d_L=2g-2+p$ with $p>0$. An $L$-twisted Higgs bundle on $X$ is a vector bundle $E$ together with an homomorphism
$$\varphi:E\longrightarrow E\otimes L$$
called Higgs field. An $L$-twisted Higgs bundle $(E,\varphi)$ is called semistable if for all subbundles $F\subseteq E$ such that $\varphi(F)\subseteq F\otimes L$ we have
$$\frac{\deg(F)}{\rk(F)} \le \frac{\deg(E)}{\rk(E)}$$
Let $\SM_L(r,d)$ denote the moduli space of semistable $L$-twisted Higgs bundles of rank $r$ and degree $d$ on $X$ (c.f. \cite{AO21}).

As an application of the previous methodology, we will compute explicit simplified formulas for the motives of some moduli spaces of rank 2 and rank 3 $L$-twisted Higgs bundles on $X$ as integer polynomials in the motives $a_1,\ldots,a_g$ and the Lefschetz motive $\LL$, where
$$a_i = \lambda^i(h^1(X))$$
and we will use this to verify computationally that some conjectural formulas from Mozgovoy \cite{Moz12}
about the motive of the moduli space of $L$-twisted Higgs bundles hold, at least for low rank, genus and degree.

More explicitly, we will apply the proposed algorithm to the following formulas from \cite[Corollary 7.7]{AO21} obtained through the Bialynicki-Birula decomposition of the variety.

\begin{theorem}[ {\cite[Corollary 7.7]{AO21}}]
\label{thm:geometricMotive}

\begin{enumerate}
\item For $r=1$
$$[\SM_{L}(X,1,d)]=[\Jac(X)\times H^0(X,L^\vee)]=\LL^{d_L+1-g}P_X(1)$$
\item For $r=2$, if $(2,d)=1$,
\begin{equation*}
\begin{split}
[\SM_{L}(X,2,d)] &= \frac{\LL^{4d_L+4-4g}\Big(P_X(1)P_X(\LL)-\LL^gP_X(1)^2\Big)}{(\LL-1)(\LL^2-1)}\\
& \ \ \ + \LL^{3d_L+2-2g}P_X(1)\sum_{d_1=1}^{\lfloor\frac{1+d_L}{2}\rfloor}\lambda^{1-2d_1+d_L}([X]).
\end{split}
\end{equation*}
\item For $r=3$, if $(3,d)=1$,
\begin{multline*}
[\SM_{L}(X,3,d)] = \frac{\LL^{9d_L+9-9g}P_X(1)}{(\LL-1)(\LL^2-1)^2(\LL^3-1)}\Big(\LL^{3g-1}(1+\LL+\LL^2)P_X(1)^2\\
-\LL^{2g-1}(1+\LL)^2P_X(1)P_X(\LL)+P_X(\LL)P_X(\LL^2)\Big)\\
+\frac{\LL^{7d_L+5-5g}P_X(1)^2}{\LL-1}\sum_{d_1=1}^{\lfloor\frac{1}{3}+\frac{d_L}{2}\rfloor}\bigg(\LL^{d_1+g}\lambda^{-2d_1+d_L}([X]+\LL^2) -\lambda^{-2d_1+d_L}([X]\LL+1)\bigg)\\
 +\frac{\LL^{7d_L+5-5g}P_X(1)^2}{\LL-1}\sum_{d_1=1}^{\lfloor\frac{2}{3}+\frac{d_L}{2}\rfloor}\bigg(\LL^{d_1+g-1}\lambda^{-2d_1+d_L+1}([X]+\LL^2) -\lambda^{-2d_1+d_L+1}([X]\LL+1)\bigg)\\
+\LL^{6d_L+3-3g}P_X(1)\sum_{d_1=1}^{d_L} \sum_{d_2= \max\{-d_L+d_1, 1 -d_1\}}^{\lfloor (1+d_L-a)/2\rfloor}\lambda^{-d_1+d_2+d_L}([X])\lambda^{1-d_1-2d_2+d_L}([X])
\end{multline*}
\end{enumerate}
\end{theorem}

and the following conjectural formula from \cite[Corollary 3]{Moz12} (in the rewritten form presented in \cite[Conjecture 7.12]{AO21}) obtained as a solution for the motivic ADHM recursion formula.

\begin{conjecture}{ {\cite[Conjecture 3]{Moz12}, \cite[Theorem 1.1 and Theorem 4.6]{MG19}, \cite[Conjecture 7.12]{AO21} }}
\label{conj:Mozgovoy}
For each integer $n\ge 1$, let
$$\SH_n(t)=\sum_{\lambda\in \SP_n} \prod_{s\in d(\lambda)} (-t^{a(s)-l(s)} \LL^{a(s)})^p t^{(1-g)(2l(s)+1)}Z_X(t^{h(s)}\LL^{a(s)}).$$
Define $H_r(t)$ for $r\ge 1$ as follows
\begin{equation*}
\begin{split}
\sum_{r\ge 1} H_r(t) T^r &= (1-t)(1-\LL t) \op{PLog} \Bigg (\sum_{n\ge 0} \SH_n(t) T^n \Bigg) \\
&=(1-t)(1-\LL t) \sum_{j\ge 1} \frac{\mu(j)}{j} \psi_j \bigg [\log \bigg ( 1+\sum_{n\ge 1}\SH_n(t) T^n \bigg) \bigg]\\
&=(1-t)(1-\LL t) \sum_{j\ge 1} \sum_{k\ge 1} \frac{(-1)^{k+1}\mu(j)}{jk}  \Bigg ( \sum_{n\ge 1} \psi_j[\SH_n(t)] T^{jn} \Bigg)^k.
\end{split}
\end{equation*}
Where $\SP_n$ denotes the set of partitions of $n$, considered as decreasing sequences of integers $\lambda=(\lambda_1\ge \lambda_2 \ge \cdots \ge \lambda_k>0)$ with $\sum_{i=1}^k \lambda_i=n$, and for each $\lambda\in \SP_n$,
$$d(\lambda)=\{(i,j)\in \ZZ^2 | 1\le i, \, 1\le j\le \lambda_i\},$$
$$a(i,j)=\lambda_i-j, \quad \quad l(i,j)=\lambda_j'-i=\max\{l|\lambda_l\ge j\}-i, \quad \quad h(i,j)=a(i,j)+l(i,j)+1,$$
and the operator $\op{PLog}$ in the formula is the Plethystic logarithm for the Adams operations associated to the opposite structure $\sigma$, defined as
$$\op{PLog}(A)=\sum_{j\ge 1} \frac{\mu(j)}{j} \psi_k[\log(A)].$$
Then $H_r(t)$ is a polynomial in $t$ and
$$[\SM_{\Lambda_\SL}(r,d)] = (-1)^{pr} \LL^{r^2(g-1)+p\frac{r(r+1)}{2}} H_r(1).$$
\end{conjecture}

Algorithm \ref{alg:motiveSimpAlg} has been applied to the previous formulas from Theorem \ref{thm:geometricMotive} and Conjecture \ref{conj:Mozgovoy} to obtain the following simplified motivic polynomials
$$P_{g,r,p}^{\op{BB}}, P_{g,r,p}^{\op{ADHM}} \in \ZZ[L,a_1,\ldots,a_g]$$
representing the respective formulas for each genus $g\ge 2$, rank $r\ge 2$ and $p>0$ in the sense of Theorem \ref{thm:exprSimplification}. As these polynomials are integral polynomials in a small set of variables, they could be compared computationally in a direct way, and the following was verified.

\begin{theorem}
\label{thm:verification}
Let $X$ be a smooth complex projective curve of genus $g$. Then
$$P_{g,r,p}^{\op{BB}} = P_{g,r,p}^{\op{ADHM}}$$
if the following conditions hold
\begin{itemize}
\item $2\le g\le 11$,
\item $1\le r\le 3$,
\item $d$ is coprime with $r$ and 
\item $1\le p\le 20$ (a.k.a. $2g-1\le -d_\SL \le 2g+18$).
\end{itemize}
In particular, as \cite[Corollary 7.7]{AO21} together with Theorems \ref{thm:exprSimplification} and \ref{thm:mainAlgorithm} prove that
$$[\SM_{L}(r,d)]=P_{g,r,p}^{\op{BB}}(\LL,h^1(X),\ldots,\lambda^g(h^1(X)))$$
then this shows that Mozgovoy's Conjecture \ref{conj:Mozgovoy} holds if such conditions are satisfied.
\end{theorem}

\begin{remark}
As a consequence of \cite[Theorem 6.7]{AO21}, the resulting polynomials $P_{g,r,p}^{\op{BB}} = P_{g,r,p}^{\op{ADHM}}$ also compute the motives for the corresponding moduli spaces of Lie algebroid connections in rank 2 and 3.
\end{remark}

\section{Annex: Motives of the Moduli space of twisted Higgs bundles in low genus}
This annex contains some examples of simplified polynomial formulas for the motives of the moduli space of $L$-twisted Higgs bundles of rank at most $3$, with $\deg(L)=2g-2+p$ and $p>0$ in low genus. These formulas have been obtained applying the proposed Algorithm \ref{alg:motiveSimpAlg} to Mozovoy's Conjectural formula \ref{conj:Mozgovoy}. The polynomials $P_{g,r,p}=P_{g,r,p}^{\op{ADHM}}=P_{g,r,p}^{\op{BB}}$ depend on $a_n=\lambda^n([h^1(X)])$ and the Lefschetz motive $\LL$. The resulting polynomial has then been factorized to improved readability. Please notice that the equations from Theorem \ref{thm:geometricMotive} imply straightforwardly that the motives of the moduli spaces are multiples of
$$P_X(1)=[\Jac(X)]=\sum_{n=0}^{2g} \lambda^n(h^1(X))=1+a_1+a_2+\ldots+a_g+\LL a_{g-1}+\ldots+\LL^{g-1} a_1+\LL^g$$
This can be observed in all the computed cases.

As these formulas have been computed from Conjecture \ref{conj:Mozgovoy}, they were, in general, just conjectural, but, as stated in Theorem \ref{thm:verification}, they have been computationally verified for $1\le p\le 20$ showing that they agree with the equations from Theorem \ref{thm:geometricMotive}.

\begin{dmath*}
P_{2,1,p}=L^{p+1}\left(L^2+a_{1}\,L+a_{1}+a_{2}+1\right)
\end{dmath*}

\begin{dmath*}
P_{3,1,p}=L^{p+2}\left(L^3+a_{1}\,L^2+a_{2}\,L+a_{1}+a_{2}+a_{3}+1\right)
\end{dmath*}

\begin{dmath*}
P_{4,1,p}=L^{p+3}\left(L^4+a_{1}\,L^3+a_{2}\,L^2+a_{3}\,L+a_{1}+a_{2}+a_{3}+a_{4}+1\right)
\end{dmath*}

\begin{dmath*}
P_{2,2,1}=L^7\left(L^2+a_{1}\,L+a_{1}+a_{2}+1\right)\,\left(2\,L+a_{1}+a_{2}+L\,a_{1}+L^2\,a_{1}+2\,L^2+L^3+L^4+2\right)
\end{dmath*}

\begin{dmath*}
P_{2,2,2}=L^{10}\left(L^2+a_{1}\,L+a_{1}+a_{2}+1\right)\,\left(2\,L+2\,a_{1}+a_{2}+2\,L\,a_{1}+L\,a_{2}+L^2\,a_{1}+L^3\,a_{1}+2\,L^2+2\,L^3+L^4+L^5+2\right)
\end{dmath*}

\begin{dmath*}
P_{2,2,3}=L^{13}\left(L^2+a_{1}\,L+a_{1}+a_{2}+1\right)\,\left(2\,L+2\,a_{1}+2\,a_{2}+3\,L\,a_{1}+L\,a_{2}+2\,L^2\,a_{1}+L^2\,a_{2}+L^3\,a_{1}+L^4\,a_{1}+3\,L^2+2\,L^3+2\,L^4+L^5+L^6+3\right)
\end{dmath*}

\begin{dmath*}
P_{2,2,4}=L^{16}\left(L^2+a_{1}\,L+a_{1}+a_{2}+1\right)\,\left(3\,L+3\,a_{1}+2\,a_{2}+4\,L\,a_{1}+2\,L\,a_{2}+3\,L^2\,a_{1}+L^2\,a_{2}+2\,L^3\,a_{1}+L^3\,a_{2}+L^4\,a_{1}+L^5\,a_{1}+3\,L^2+3\,L^3+2\,L^4+2\,L^5+L^6+L^7+3\right)
\end{dmath*}

\begin{dmath*}
P_{2,3,1}=L^{15}\left(L^2+a_{1}\,L+a_{1}+a_{2}+1\right)\,\left(L^{11}+L^{10}+L^9\,a_{1}+3\,L^9+2\,L^8\,a_{1}+4\,L^8+4\,L^7\,a_{1}+L^7\,a_{2}+7\,L^7+L^6\,{a_{1}}^2+8\,L^6\,a_{1}+L^6\,a_{2}+9\,L^6+L^5\,{a_{1}}^2+12\,L^5\,a_{1}+4\,L^5\,a_{2}+14\,L^5+3\,L^4\,{a_{1}}^2+L^4\,a_{1}\,a_{2}+18\,L^4\,a_{1}+5\,L^4\,a_{2}+15\,L^4+5\,L^3\,{a_{1}}^2+2\,L^3\,a_{1}\,a_{2}+22\,L^3\,a_{1}+9\,L^3\,a_{2}+18\,L^3+6\,L^2\,{a_{1}}^2+4\,L^2\,a_{1}\,a_{2}+24\,L^2\,a_{1}+10\,L^2\,a_{2}+15\,L^2+6\,L\,{a_{1}}^2+5\,L\,a_{1}\,a_{2}+18\,L\,a_{1}+L\,{a_{2}}^2+10\,L\,a_{2}+12\,L+3\,{a_{1}}^2+4\,a_{1}\,a_{2}+9\,a_{1}+{a_{2}}^2+6\,a_{2}+6\right)
\end{dmath*}

\begin{dmath*}
P_{2,3,2}=L^{21}\left(L^2+a_{1}\,L+a_{1}+a_{2}+1\right)\,\left(L^{14}+L^{13}+L^{12}\,a_{1}+3\,L^{12}+2\,L^{11}\,a_{1}+4\,L^{11}+4\,L^{10}\,a_{1}+L^{10}\,a_{2}+7\,L^{10}+L^9\,{a_{1}}^2+8\,L^9\,a_{1}+L^9\,a_{2}+9\,L^9+L^8\,{a_{1}}^2+12\,L^8\,a_{1}+4\,L^8\,a_{2}+15\,L^8+3\,L^7\,{a_{1}}^2+L^7\,a_{1}\,a_{2}+19\,L^7\,a_{1}+5\,L^7\,a_{2}+18\,L^7+5\,L^6\,{a_{1}}^2+2\,L^6\,a_{1}\,a_{2}+28\,L^6\,a_{1}+10\,L^6\,a_{2}+25\,L^6+8\,L^5\,{a_{1}}^2+4\,L^5\,a_{1}\,a_{2}+38\,L^5\,a_{1}+13\,L^5\,a_{2}+28\,L^5+13\,L^4\,{a_{1}}^2+7\,L^4\,a_{1}\,a_{2}+44\,L^4\,a_{1}+L^4\,{a_{2}}^2+19\,L^4\,a_{2}+32\,L^4+16\,L^3\,{a_{1}}^2+12\,L^3\,a_{1}\,a_{2}+50\,L^3\,a_{1}+L^3\,{a_{2}}^2+21\,L^3\,a_{2}+29\,L^3+17\,L^2\,{a_{1}}^2+15\,L^2\,a_{1}\,a_{2}+44\,L^2\,a_{1}+3\,L^2\,{a_{2}}^2+23\,L^2\,a_{2}+27\,L^2+15\,L\,{a_{1}}^2+16\,L\,a_{1}\,a_{2}+32\,L\,a_{1}+4\,L\,{a_{2}}^2+17\,L\,a_{2}+16\,L+6\,{a_{1}}^2+9\,a_{1}\,a_{2}+16\,a_{1}+3\,{a_{2}}^2+12\,a_{2}+10\right)
\end{dmath*}

\begin{dmath*}
P_{2,3,3}=L^{27}\left(L^2+a_{1}\,L+a_{1}+a_{2}+1\right)\,\left(L^{17}+L^{16}+L^{15}\,a_{1}+3\,L^{15}+2\,L^{14}\,a_{1}+4\,L^{14}+4\,L^{13}\,a_{1}+L^{13}\,a_{2}+7\,L^{13}+L^{12}\,{a_{1}}^2+8\,L^{12}\,a_{1}+L^{12}\,a_{2}+9\,L^{12}+L^{11}\,{a_{1}}^2+12\,L^{11}\,a_{1}+4\,L^{11}\,a_{2}+15\,L^{11}+3\,L^{10}\,{a_{1}}^2+L^{10}\,a_{1}\,a_{2}+19\,L^{10}\,a_{1}+5\,L^{10}\,a_{2}+18\,L^{10}+5\,L^9\,{a_{1}}^2+2\,L^9\,a_{1}\,a_{2}+28\,L^9\,a_{1}+10\,L^9\,a_{2}+27\,L^9+8\,L^8\,{a_{1}}^2+4\,L^8\,a_{1}\,a_{2}+40\,L^8\,a_{1}+13\,L^8\,a_{2}+33\,L^8+13\,L^7\,{a_{1}}^2+7\,L^7\,a_{1}\,a_{2}+54\,L^7\,a_{1}+L^7\,{a_{2}}^2+21\,L^7\,a_{2}+43\,L^7+19\,L^6\,{a_{1}}^2+12\,L^6\,a_{1}\,a_{2}+71\,L^6\,a_{1}+L^6\,{a_{2}}^2+26\,L^6\,a_{2}+48\,L^6+27\,L^5\,{a_{1}}^2+18\,L^5\,a_{1}\,a_{2}+82\,L^5\,a_{1}+3\,L^5\,{a_{2}}^2+36\,L^5\,a_{2}+55\,L^5+34\,L^4\,{a_{1}}^2+27\,L^4\,a_{1}\,a_{2}+92\,L^4\,a_{1}+4\,L^4\,{a_{2}}^2+40\,L^4\,a_{2}+53\,L^4+38\,L^3\,{a_{1}}^2+34\,L^3\,a_{1}\,a_{2}+91\,L^3\,a_{1}+7\,L^3\,{a_{2}}^2+45\,L^3\,a_{2}+51\,L^3+38\,L^2\,{a_{1}}^2+38\,L^2\,a_{1}\,a_{2}+78\,L^2\,a_{1}+9\,L^2\,{a_{2}}^2+41\,L^2\,a_{2}+41\,L^2+28\,L\,{a_{1}}^2+33\,L\,a_{1}\,a_{2}+55\,L\,a_{1}+9\,L\,{a_{2}}^2+31\,L\,a_{2}+25\,L+10\,{a_{1}}^2+16\,a_{1}\,a_{2}+25\,a_{1}+6\,{a_{2}}^2+20\,a_{2}+15\right)
\end{dmath*}

\begin{dmath*}
P_{2,3,4}=L^{33}\left(L^2+a_{1}\,L+a_{1}+a_{2}+1\right)\,\left(L^{20}+L^{19}+L^{18}\,a_{1}+3\,L^{18}+2\,L^{17}\,a_{1}+4\,L^{17}+4\,L^{16}\,a_{1}+L^{16}\,a_{2}+7\,L^{16}+L^{15}\,{a_{1}}^2+8\,L^{15}\,a_{1}+L^{15}\,a_{2}+9\,L^{15}+L^{14}\,{a_{1}}^2+12\,L^{14}\,a_{1}+4\,L^{14}\,a_{2}+15\,L^{14}+3\,L^{13}\,{a_{1}}^2+L^{13}\,a_{1}\,a_{2}+19\,L^{13}\,a_{1}+5\,L^{13}\,a_{2}+18\,L^{13}+5\,L^{12}\,{a_{1}}^2+2\,L^{12}\,a_{1}\,a_{2}+28\,L^{12}\,a_{1}+10\,L^{12}\,a_{2}+27\,L^{12}+8\,L^{11}\,{a_{1}}^2+4\,L^{11}\,a_{1}\,a_{2}+40\,L^{11}\,a_{1}+13\,L^{11}\,a_{2}+33\,L^{11}+13\,L^{10}\,{a_{1}}^2+7\,L^{10}\,a_{1}\,a_{2}+54\,L^{10}\,a_{1}+L^{10}\,{a_{2}}^2+21\,L^{10}\,a_{2}+46\,L^{10}+19\,L^9\,{a_{1}}^2+12\,L^9\,a_{1}\,a_{2}+74\,L^9\,a_{1}+L^9\,{a_{2}}^2+26\,L^9\,a_{2}+55\,L^9+27\,L^8\,{a_{1}}^2+18\,L^8\,a_{1}\,a_{2}+96\,L^8\,a_{1}+3\,L^8\,{a_{2}}^2+39\,L^8\,a_{2}+70\,L^8+38\,L^7\,{a_{1}}^2+27\,L^7\,a_{1}\,a_{2}+120\,L^7\,a_{1}+4\,L^7\,{a_{2}}^2+47\,L^7\,a_{2}+78\,L^7+51\,L^6\,{a_{1}}^2+38\,L^6\,a_{1}\,a_{2}+140\,L^6\,a_{1}+7\,L^6\,{a_{2}}^2+62\,L^6\,a_{2}+88\,L^6+62\,L^5\,{a_{1}}^2+52\,L^5\,a_{1}\,a_{2}+158\,L^5\,a_{1}+9\,L^5\,{a_{2}}^2+70\,L^5\,a_{2}+89\,L^5+71\,L^4\,{a_{1}}^2+64\,L^4\,a_{1}\,a_{2}+162\,L^4\,a_{1}+14\,L^4\,{a_{2}}^2+79\,L^4\,a_{2}+90\,L^4+75\,L^3\,{a_{1}}^2+73\,L^3\,a_{1}\,a_{2}+156\,L^3\,a_{1}+17\,L^3\,{a_{2}}^2+78\,L^3\,a_{2}+78\,L^3+67\,L^2\,{a_{1}}^2+71\,L^2\,a_{1}\,a_{2}+128\,L^2\,a_{1}+18\,L^2\,{a_{2}}^2+71\,L^2\,a_{2}+64\,L^2+45\,L\,{a_{1}}^2+56\,L\,a_{1}\,a_{2}+84\,L\,a_{1}+16\,L\,{a_{2}}^2+49\,L\,a_{2}+36\,L+15\,{a_{1}}^2+25\,a_{1}\,a_{2}+36\,a_{1}+10\,{a_{2}}^2+30\,a_{2}+21\right)
\end{dmath*}

\begin{dmath*}
P_{3,2,1}=L^{11}\left(L^3+a_{1}\,L^2+a_{2}\,L+a_{1}+a_{2}+a_{3}+1\right)\,\left(3\,L+2\,a_{1}+2\,a_{2}+a_{3}+2\,L\,a_{1}+2\,L\,a_{2}+L\,a_{3}+2\,L^2\,a_{1}+L^2\,a_{2}+2\,L^3\,a_{1}+L^3\,a_{2}+L^4\,a_{1}+L^5\,a_{1}+3\,L^2+3\,L^3+2\,L^4+2\,L^5+L^6+L^7+3\right)
\end{dmath*}

\begin{dmath*}
P_{3,2,2}=L^{14}\left(L^3+a_{1}\,L^2+a_{2}\,L+a_{1}+a_{2}+a_{3}+1\right)\,\left(3\,L+3\,a_{1}+2\,a_{2}+2\,a_{3}+2\,L\,a_{1}+3\,L\,a_{2}+L\,a_{3}+3\,L^2\,a_{1}+2\,L^2\,a_{2}+2\,L^3\,a_{1}+L^2\,a_{3}+L^3\,a_{2}+2\,L^4\,a_{1}+L^4\,a_{2}+L^5\,a_{1}+L^6\,a_{1}+3\,L^2+3\,L^3+3\,L^4+2\,L^5+2\,L^6+L^7+L^8+3\right)
\end{dmath*}

\begin{dmath*}
P_{3,2,3}=L^{17}\left(L^3+a_{1}\,L^2+a_{2}\,L+a_{1}+a_{2}+a_{3}+1\right)\,\left(3\,L+3\,a_{1}+3\,a_{2}+2\,a_{3}+3\,L\,a_{1}+4\,L\,a_{2}+2\,L\,a_{3}+3\,L^2\,a_{1}+3\,L^2\,a_{2}+3\,L^3\,a_{1}+L^2\,a_{3}+2\,L^3\,a_{2}+2\,L^4\,a_{1}+L^3\,a_{3}+L^4\,a_{2}+2\,L^5\,a_{1}+L^5\,a_{2}+L^6\,a_{1}+L^7\,a_{1}+3\,L^2+4\,L^3+3\,L^4+3\,L^5+2\,L^6+2\,L^7+L^8+L^9+4\right)
\end{dmath*}

\begin{dmath*}
P_{3,2,4}=L^{20}\left(L^3+a_{1}\,L^2+a_{2}\,L+a_{1}+a_{2}+a_{3}+1\right)\,\left(4\,L+4\,a_{1}+3\,a_{2}+3\,a_{3}+3\,L\,a_{1}+5\,L\,a_{2}+2\,L\,a_{3}+5\,L^2\,a_{1}+4\,L^2\,a_{2}+3\,L^3\,a_{1}+2\,L^2\,a_{3}+3\,L^3\,a_{2}+3\,L^4\,a_{1}+L^3\,a_{3}+2\,L^4\,a_{2}+2\,L^5\,a_{1}+L^4\,a_{3}+L^5\,a_{2}+2\,L^6\,a_{1}+L^6\,a_{2}+L^7\,a_{1}+L^8\,a_{1}+3\,L^2+4\,L^3+4\,L^4+3\,L^5+3\,L^6+2\,L^7+2\,L^8+L^9+L^{10}+4\right)
\end{dmath*}

\begin{dmath*}
P_{3,3,1}=L^{24}\left(L^3+a_{1}\,L^2+a_{2}\,L+a_{1}+a_{2}+a_{3}+1\right)\,\left(L^{19}+L^{18}+L^{17}\,a_{1}+3\,L^{17}+2\,L^{16}\,a_{1}+4\,L^{16}+4\,L^{15}\,a_{1}+L^{15}\,a_{2}+7\,L^{15}+L^{14}\,{a_{1}}^2+7\,L^{14}\,a_{1}+L^{14}\,a_{2}+9\,L^{14}+L^{13}\,{a_{1}}^2+11\,L^{13}\,a_{1}+4\,L^{13}\,a_{2}+L^{13}\,a_{3}+14\,L^{13}+3\,L^{12}\,{a_{1}}^2+L^{12}\,a_{1}\,a_{2}+16\,L^{12}\,a_{1}+6\,L^{12}\,a_{2}+L^{12}\,a_{3}+17\,L^{12}+4\,L^{11}\,{a_{1}}^2+2\,L^{11}\,a_{1}\,a_{2}+24\,L^{11}\,a_{1}+11\,L^{11}\,a_{2}+3\,L^{11}\,a_{3}+24\,L^{11}+7\,L^{10}\,{a_{1}}^2+4\,L^{10}\,a_{1}\,a_{2}+L^{10}\,a_{1}\,a_{3}+32\,L^{10}\,a_{1}+16\,L^{10}\,a_{2}+5\,L^{10}\,a_{3}+30\,L^{10}+9\,L^9\,{a_{1}}^2+8\,L^9\,a_{1}\,a_{2}+L^9\,a_{1}\,a_{3}+44\,L^9\,a_{1}+L^9\,{a_{2}}^2+25\,L^9\,a_{2}+8\,L^9\,a_{3}+38\,L^9+15\,L^8\,{a_{1}}^2+12\,L^8\,a_{1}\,a_{2}+4\,L^8\,a_{1}\,a_{3}+56\,L^8\,a_{1}+L^8\,{a_{2}}^2+34\,L^8\,a_{2}+12\,L^8\,a_{3}+45\,L^8+18\,L^7\,{a_{1}}^2+19\,L^7\,a_{1}\,a_{2}+5\,L^7\,a_{1}\,a_{3}+70\,L^7\,a_{1}+3\,L^7\,{a_{2}}^2+L^7\,a_{2}\,a_{3}+47\,L^7\,a_{2}+18\,L^7\,a_{3}+53\,L^7+25\,L^6\,{a_{1}}^2+28\,L^6\,a_{1}\,a_{2}+10\,L^6\,a_{1}\,a_{3}+80\,L^6\,a_{1}+5\,L^6\,{a_{2}}^2+2\,L^6\,a_{2}\,a_{3}+59\,L^6\,a_{2}+23\,L^6\,a_{3}+57\,L^6+28\,L^5\,{a_{1}}^2+38\,L^5\,a_{1}\,a_{2}+13\,L^5\,a_{1}\,a_{3}+90\,L^5\,a_{1}+8\,L^5\,{a_{2}}^2+4\,L^5\,a_{2}\,a_{3}+69\,L^5\,a_{2}+30\,L^5\,a_{3}+60\,L^5+32\,L^4\,{a_{1}}^2+44\,L^4\,a_{1}\,a_{2}+19\,L^4\,a_{1}\,a_{3}+89\,L^4\,a_{1}+13\,L^4\,{a_{2}}^2+7\,L^4\,a_{2}\,a_{3}+76\,L^4\,a_{2}+L^4\,{a_{3}}^2+35\,L^4\,a_{3}+58\,L^4+29\,L^3\,{a_{1}}^2+50\,L^3\,a_{1}\,a_{2}+21\,L^3\,a_{1}\,a_{3}+85\,L^3\,a_{1}+16\,L^3\,{a_{2}}^2+12\,L^3\,a_{2}\,a_{3}+76\,L^3\,a_{2}+L^3\,{a_{3}}^2+36\,L^3\,a_{3}+51\,L^3+27\,L^2\,{a_{1}}^2+44\,L^2\,a_{1}\,a_{2}+23\,L^2\,a_{1}\,a_{3}+67\,L^2\,a_{1}+17\,L^2\,{a_{2}}^2+15\,L^2\,a_{2}\,a_{3}+64\,L^2\,a_{2}+3\,L^2\,{a_{3}}^2+33\,L^2\,a_{3}+39\,L^2+16\,L\,{a_{1}}^2+32\,L\,a_{1}\,a_{2}+17\,L\,a_{1}\,a_{3}+45\,L\,a_{1}+15\,L\,{a_{2}}^2+16\,L\,a_{2}\,a_{3}+47\,L\,a_{2}+4\,L\,{a_{3}}^2+28\,L\,a_{3}+30\,L+10\,{a_{1}}^2+16\,a_{1}\,a_{2}+12\,a_{1}\,a_{3}+25\,a_{1}+6\,{a_{2}}^2+9\,a_{2}\,a_{3}+20\,a_{2}+3\,{a_{3}}^2+15\,a_{3}+15\right)
\end{dmath*}

\begin{dmath*}
P_{3,3,2}=L^{30}\left(L^3+a_{1}\,L^2+a_{2}\,L+a_{1}+a_{2}+a_{3}+1\right)\,\left(L^{22}+L^{21}+L^{20}\,a_{1}+3\,L^{20}+2\,L^{19}\,a_{1}+4\,L^{19}+4\,L^{18}\,a_{1}+L^{18}\,a_{2}+7\,L^{18}+L^{17}\,{a_{1}}^2+7\,L^{17}\,a_{1}+L^{17}\,a_{2}+9\,L^{17}+L^{16}\,{a_{1}}^2+11\,L^{16}\,a_{1}+4\,L^{16}\,a_{2}+L^{16}\,a_{3}+14\,L^{16}+3\,L^{15}\,{a_{1}}^2+L^{15}\,a_{1}\,a_{2}+16\,L^{15}\,a_{1}+6\,L^{15}\,a_{2}+L^{15}\,a_{3}+17\,L^{15}+4\,L^{14}\,{a_{1}}^2+2\,L^{14}\,a_{1}\,a_{2}+24\,L^{14}\,a_{1}+11\,L^{14}\,a_{2}+3\,L^{14}\,a_{3}+24\,L^{14}+7\,L^{13}\,{a_{1}}^2+4\,L^{13}\,a_{1}\,a_{2}+L^{13}\,a_{1}\,a_{3}+32\,L^{13}\,a_{1}+16\,L^{13}\,a_{2}+5\,L^{13}\,a_{3}+30\,L^{13}+9\,L^{12}\,{a_{1}}^2+8\,L^{12}\,a_{1}\,a_{2}+L^{12}\,a_{1}\,a_{3}+44\,L^{12}\,a_{1}+L^{12}\,{a_{2}}^2+25\,L^{12}\,a_{2}+8\,L^{12}\,a_{3}+39\,L^{12}+15\,L^{11}\,{a_{1}}^2+12\,L^{11}\,a_{1}\,a_{2}+4\,L^{11}\,a_{1}\,a_{3}+57\,L^{11}\,a_{1}+L^{11}\,{a_{2}}^2+34\,L^{11}\,a_{2}+12\,L^{11}\,a_{3}+48\,L^{11}+18\,L^{10}\,{a_{1}}^2+19\,L^{10}\,a_{1}\,a_{2}+5\,L^{10}\,a_{1}\,a_{3}+75\,L^{10}\,a_{1}+3\,L^{10}\,{a_{2}}^2+L^{10}\,a_{2}\,a_{3}+48\,L^{10}\,a_{2}+18\,L^{10}\,a_{3}+59\,L^{10}+27\,L^9\,{a_{1}}^2+28\,L^9\,a_{1}\,a_{2}+10\,L^9\,a_{1}\,a_{3}+92\,L^9\,a_{1}+5\,L^9\,{a_{2}}^2+2\,L^9\,a_{2}\,a_{3}+63\,L^9\,a_{2}+24\,L^9\,a_{3}+68\,L^9+33\,L^8\,{a_{1}}^2+40\,L^8\,a_{1}\,a_{2}+13\,L^8\,a_{1}\,a_{3}+112\,L^8\,a_{1}+8\,L^8\,{a_{2}}^2+4\,L^8\,a_{2}\,a_{3}+81\,L^8\,a_{2}+33\,L^8\,a_{3}+78\,L^8+43\,L^7\,{a_{1}}^2+54\,L^7\,a_{1}\,a_{2}+21\,L^7\,a_{1}\,a_{3}+126\,L^7\,a_{1}+13\,L^7\,{a_{2}}^2+7\,L^7\,a_{2}\,a_{3}+99\,L^7\,a_{2}+L^7\,{a_{3}}^2+41\,L^7\,a_{3}+85\,L^7+48\,L^6\,{a_{1}}^2+71\,L^6\,a_{1}\,a_{2}+26\,L^6\,a_{1}\,a_{3}+140\,L^6\,a_{1}+19\,L^6\,{a_{2}}^2+12\,L^6\,a_{2}\,a_{3}+113\,L^6\,a_{2}+L^6\,{a_{3}}^2+50\,L^6\,a_{3}+88\,L^6+55\,L^5\,{a_{1}}^2+82\,L^5\,a_{1}\,a_{2}+36\,L^5\,a_{1}\,a_{3}+142\,L^5\,a_{1}+27\,L^5\,{a_{2}}^2+18\,L^5\,a_{2}\,a_{3}+124\,L^5\,a_{2}+3\,L^5\,{a_{3}}^2+57\,L^5\,a_{3}+87\,L^5+53\,L^4\,{a_{1}}^2+92\,L^4\,a_{1}\,a_{2}+40\,L^4\,a_{1}\,a_{3}+139\,L^4\,a_{1}+34\,L^4\,{a_{2}}^2+27\,L^4\,a_{2}\,a_{3}+127\,L^4\,a_{2}+4\,L^4\,{a_{3}}^2+61\,L^4\,a_{3}+81\,L^4+51\,L^3\,{a_{1}}^2+91\,L^3\,a_{1}\,a_{2}+45\,L^3\,a_{1}\,a_{3}+122\,L^3\,a_{1}+38\,L^3\,{a_{2}}^2+34\,L^3\,a_{2}\,a_{3}+115\,L^3\,a_{2}+7\,L^3\,{a_{3}}^2+58\,L^3\,a_{3}+69\,L^3+41\,L^2\,{a_{1}}^2+78\,L^2\,a_{1}\,a_{2}+41\,L^2\,a_{1}\,a_{3}+94\,L^2\,a_{1}+38\,L^2\,{a_{2}}^2+38\,L^2\,a_{2}\,a_{3}+94\,L^2\,a_{2}+9\,L^2\,{a_{3}}^2+51\,L^2\,a_{3}+51\,L^2+25\,L\,{a_{1}}^2+55\,L\,a_{1}\,a_{2}+31\,L\,a_{1}\,a_{3}+60\,L\,a_{1}+28\,L\,{a_{2}}^2+33\,L\,a_{2}\,a_{3}+68\,L\,a_{2}+9\,L\,{a_{3}}^2+39\,L\,a_{3}+36\,L+15\,{a_{1}}^2+25\,a_{1}\,a_{2}+20\,a_{1}\,a_{3}+36\,a_{1}+10\,{a_{2}}^2+16\,a_{2}\,a_{3}+30\,a_{2}+6\,{a_{3}}^2+24\,a_{3}+21\right)
\end{dmath*}
\bibliographystyle{alpha}
\bibliography{MotiveComputationsAlfaya}

\begin{thebibliography}{GPHS14}

\bibitem[AO21]{AO21}
David Alfaya and Andr{\'e} Oliveira.
\newblock {Lie} algebroid connections, twisted {Higgs} bundles and motives of
  moduli spaces.
\newblock {\em arXiv:2102.12246}, 2021.

\bibitem[BD07]{BD07}
Kai Behrend and Ajneet Dhillon.
\newblock On the motivic class of the stack of bundles.
\newblock {\em Advances in Mathematics}, 212(2):617--644, 2007.

\bibitem[GL20]{GL20}
Tom{\'a}s~L. G{\'o}mez and Kyoung-Seog Lee.
\newblock Motivic decompositions of moduli spaces of vector bundles on curves.
\newblock {\em arXiv:2007.06067}, 2020.

\bibitem[GPHS14]{GPHS11}
Oscar Garc{\'i}a-Prada, Jochen Heinloth, and Alexander Schmitt.
\newblock On the motives of moduli of chains and {Higgs} bundles.
\newblock {\em Journal of the European Mathematical Society}, 16:2617--2668,
  2014.

\bibitem[Gri19]{Gri19}
Darij Grinberg.
\newblock {$\lambda$}-rings: Definitions and basic properties, 2019.
\newblock https://www.cip.ifi.lmu.de/~grinberg/algebra/lambda.pdf.

\bibitem[Hei07]{H07}
Franziska Heinloth.
\newblock A note on functional equations for zeta functions with values in
  {C}how motives.
\newblock {\em Ann. Inst. Fourier (Grenoble)}, 57(6):1927--1945, 2007.

\bibitem[Kap00]{Kapranov00}
M.~Kapranov.
\newblock The elliptic curve in the {S}-duality theory and {Eisenstein} series
  for {Kac-Moody} groups.
\newblock {\em arXiv:math/0001005}, 2000.

\bibitem[Knu73]{Knut73}
Donald Knutson.
\newblock {\em $\lambda$-Rings and the Representation Theory of the Symmetric
  Group}.
\newblock Springer Berlin Heidelberg, Berlin, Heidelberg, 1973.

\bibitem[Lee18]{Lee18}
Kyoung-Seog Lee.
\newblock Remarks on motives of moduli spaces of rank 2 vector bundles on
  curves.
\newblock {\em arXiv:1806.11101}, 2018.

\bibitem[MO19]{MG19}
Sergey Mozgovoy and Ronan O'Gorman.
\newblock Counting twisted {H}iggs bundles.
\newblock {\em arXiv:1901.02439}, 2019.

\bibitem[Moz12]{Moz12}
Sergey Mozgovoy.
\newblock Solutions of the motivic {ADHM} recursion formula.
\newblock {\em Int. Math. Res. Not. IMRN}, 2012(18):4218--4244, 2012.

\bibitem[S{\'a}n14]{S14}
Jonathan S{\'a}nchez.
\newblock {\em Motives of moduli spaces of pairs and applications}.
\newblock PhD thesis, Universidad Complutense, Madrid, 2014.

\end{thebibliography}

\end{document}